\documentclass[pdflatex,sn-mathphys-num]{sn-jnl}% Math and Physical Sciences Numbered Reference Style
\usepackage{graphicx}%
\usepackage{multirow}%
\usepackage{amsmath,amssymb,amsfonts}%
\usepackage{amsthm}%
\usepackage{mathrsfs}%
\usepackage[title]{appendix}%
\usepackage{xcolor}%
\usepackage{textcomp}%
\usepackage{manyfoot}%
\usepackage{booktabs}%
\usepackage{algorithm}%
\usepackage{algorithmicx}%
\usepackage{algpseudocode}%
\usepackage{listings}%
\usepackage{subcaption}
\numberwithin{equation}{section}

\newcommand{\ii}{\mathrm{i}}

\newcommand{\real}{\mathbb{R}}
\newcommand{\ints}{\mathbb{Z}}
\newcommand{\natu}{\mathbb{N}}

\newcommand{\bsk}{\boldsymbol{k}}

\newcommand{\bsu}{\boldsymbol{u}}

\newcommand{\bsx}{\boldsymbol{x}}
\newcommand{\bsy}{\boldsymbol{y}}
\newcommand{\bst}{\boldsymbol{t}}
\newcommand{\bsT}{\boldsymbol{T}}

\newcommand{\bsz}{\boldsymbol{z}}

\newcommand{\bsU}{\boldsymbol{U}}

\newcommand{\bsgamma}{\boldsymbol{\gamma}}

\newcommand{\rd}{\,\mathrm{d}}

\newcommand{\e}{\mathbb{E}}

\newcommand{\supp}{\boldsymbol{s}}

\newcommand{\mix}{\mathrm{mix}}
\newcommand{\itr}{\mathbb{I}}

\theoremstyle{thmstyleone}%
\newtheorem{theorem}{Theorem}[section]% meant for sectionwise numbers
\newtheorem{lemma}[theorem]{Lemma}
\newtheorem{assumption}[theorem]{Assumption}
\newtheorem{corollary}[theorem]{Corollary}

\newtheorem{remark}[theorem]{Remark}%

\begin{document}

\title[QMC importance sampling]{Quasi-Monte Carlo integration over $\mathbb{R}^s$ with boundary-damping importance sampling}


\author[1,2]{\fnm{Zexin} \sur{Pan}}\email{zep002@zju.edu.cn}

\author[3,4]{\fnm{Du} \sur{Ouyang}}\email{du.ouyang@polytechnique.edu}

\author*[5]{\fnm{Zhijian} \sur{He}}\email{hezhijian@scut.edu.cn}


\affil[1]{\orgdiv{Institute of Fundamental and Transdisciplinary Research}, \orgname{Zhejiang University}, \orgaddress{\city{Hangzhou}, \postcode{310058}, \country{China}}}

\affil[2]{\orgdiv{Johann Radon Institute for Computational and Applied Mathematics (RICAM)}, \orgname{Austrian Academy of Sciences}, \orgaddress{\city{Linz}, \postcode{4040}, \country{Austria}}}

\affil[3]{\orgdiv{Centre de Math\'{e}matiques Appliqu\'{e}es}, \orgname{\'{E}cole Polytechnique}, \orgaddress{\city{Palaiseau}, \postcode{91128}, \country{France}}}
\affil[4]{\orgdiv{Department of Mathematical Sciences}, \orgname{Tsinghua University}, \orgaddress{\city{Beijing}, \postcode{100084}, \country{China}}}

\affil[5]{\orgdiv{School of Mathematics}, \orgname{South China University of Technology}, \orgaddress{\city{Guangzhou}, \postcode{510641}, \country{China}}}


\abstract{This paper proposes a new importance sampling (IS) method tailored to quasi-Monte Carlo (QMC) integration over $\mathbb{R}^s$. IS introduces a multiplicative weight that compensates for sampling from a proposal distribution rather than the target distribution. An unsuitable proposal, however, can yield an integrand that exhibits severe boundary growth and suffers from the curse of dimensionality, thus degrading QMC performance. Our approach is to design a weight function with the desired regularity and then construct a transport map realizing it as an IS proposal, thereby obtaining a transformed integrand amenable to QMC integration. We provide a rigorous analysis based on Fourier coefficients and apply it to lattice rule integration. Under suitable conditions, the proposed estimator achieves a fast convergence rate independent of the dimension $s$, making it suitable for high-dimensional problems.}

\keywords{Quasi-Monte Carlo, importance sampling, transport maps}


\pacs[MSC Classification]{41A63, 65D30, 97N40}

\maketitle
\section{Introduction}
Quasi-Monte Carlo (QMC) is an effective method for high-dimensional integration and has been successfully applied to many fields, including computational finance \cite{lecuyer:2009,zhang:2021} and uncertainty quantification \cite{schillings:2020,kaarnioja:2022}. While QMC can achieve a faster convergence rate than plain Monte Carlo, its gain depends on the regularity of the integrand and on the choice of sampling proposal, with the latter specified by a transport map that pushes QMC quadrature points on the unit cube to the proposal distribution. For integrands on $\real^s$ with unbounded growth at infinity, the standard inversion proposal leads to singularities in the transformed integrand near the boundary of the unit cube, which may significantly degrade the QMC convergence rate.

Importance sampling (IS) changes the sampling measure by replacing the target distribution with a tractable proposal distribution. It is not only widely used as a variance reduction technique in the Monte Carlo setting \cite{owen:2013}, but is also used in Bayesian statistics and Bayesian inverse problems to approximate the target distribution using weighted samples drawn from suitable proposal distributions \cite{agapiou:2017}. 

Recently, IS has been applied to QMC integration of unbounded integrands, successfully improving the asymptotic convergence rates by smoothing the transformed integrand near the boundary \cite{zhang:2021,he:2023,ouyang:2024}. However, it also introduces a dimensionality effect, causing the error bounds to grow exponentially with the dimension. 

In this paper, we propose a new IS approach that retains the boundary-smoothing benefit while keeping the dimensionality effect under control. As a result, our method attains a fast convergence rate independent of the dimension, a property known as strong tractability in information-based complexity \cite{nova:wozn:2008}.

To formalize our approach, we begin by introducing the notation and the class of IS estimators under consideration.
Throughout this paper, we let $\real=(-\infty,\infty)$ and $\itr=(0,1)$. We use $\bsx$ for coordinates in $\real^s$ and $\bsu$ for coordinates in $\itr^s$. We consider integrals of the form
\begin{equation}\label{eq:integral}
    \mu(f)=\int_{\real^s} f(\bsx) \prod_{j=1}^s \varphi(x_j)\rd \bsx,
\end{equation}
where $f:\real^s \to \real$ is a real-valued integrand and $\varphi:\real\to \real$ is a probability density function. Let $\Phi(x)=\int_{-\infty}^x \varphi (y) \rd y$ be the cumulative distribution function (CDF) and $\Phi^{-1}:\itr \to \real$ be the inverse CDF (quantile function). QMC integration takes quadrature points in the unit cube. Before using QMC, the sampling proposal needs to be expressed as a transformation of the standard uniform distribution, say $T:\itr^d\to \real^s$. The transformation $T$ is also called a transport map in the literature \cite{liu2024transportquasimontecarlo}. The dimension $d$ need not agree with $s$, but for simplicity we take $d=s$ in this paper.  To estimate $\mu(f)$ in \eqref{eq:integral}, we consider an estimator of the form
\begin{equation}\label{eqn:muhat}
 \widehat{\mu}(f)=\frac{1}{n}\sum_{i=0}^{n-1} w(\bsu_i)  f\circ \bsT(\bsu_i) \quad \text{for} \quad \bsT(\bsu)=(T_1(u_1),\dots, T_s(u_s)),   
\end{equation}
where $\{\bsu_0,\dots,\bsu_{n-1}\}\subseteq \itr^s$ is a (randomized) QMC point set \cite{vLEC02a,owen:2013}, the transport map $T_j:\itr\to \real$ is differentiable, and the weight function
\begin{equation*}
    w(\bsu)=\prod_{j=1}^s w_j(u_j)  \quad \text{for} \quad w_j(u)=T_j'(u)  \varphi\circ T_j(u) .
\end{equation*}
We can view \eqref{eqn:muhat} as an IS estimator whose proposal distribution has independent marginals, and interpret $w(\bsu)$ as the likelihood ratio.
Many QMC methods in the literature share this form. Examples are:
\begin{itemize}
    \item $T_j(u)=\Phi^{-1}(u)$, commonly known as inversion methods \cite{haltavoid, nichols:2014, dick2025quasi, ye2025medianqmcmethodunbounded}. For this case, $w(\bsu)=1$.
    \item $T_j(u)=a_j u+ b_j(1-u)$ for $a_j<b_j$, commonly known as truncation methods \cite{dick2018optimal,  kazashi2023suboptimality,nuyens2023scaled,goda2024randomizing, kazashi2025optimalityquasimontecarlomethods}.
    \item $T_j(u)=\Phi^{-1}_\nu(u)$ with $\Phi^{-1}_\nu(u)$ the inverse CDF of a Student's $t$-distribution with $\nu$ degree of freedom \cite{ouyang:2024}. A recent paper \cite{suzuki2025mobius} considers the M{\"o}bius transformation $T_j(u)=-\cot(\pi u)$, which can be viewed as a special case since $\cot(\pi u)$ is the inverse CDF of a Cauchy distribution.
\end{itemize}
We also note that there are many methods beyond the above form \cite{dick2011quasi, he:2023, dung2024optimal,  liu2024transportquasimontecarlo, wang2025convergence}.

Our main contributions are threefold. First, we propose a novel IS approach that improves the performance of QMC by tempering the boundary growth of the integrand. Unlike Laplace-based IS \cite{schillings:2020}, which uses Gaussian proposals, our proposal is not a standard distribution, yet it is easy to sample from. Second, we provide a rigorous analysis based on Fourier coefficients and apply it to lattice rule integration \cite{dick2022lattice}. Third, under suitable conditions, the proposed IS estimator achieves a faster convergence rate than Monte Carlo while breaking the curse of dimensionality, making it suitable for high-dimensional problems.

The rest of this paper is organized as follows. Section~\ref{sec:method} describes our new IS proposal. Section~\ref{sec:back} provides preliminaries on function norms, ANOVA decomposition, and lattice rules. Section~\ref{sec:norm} investigates the smoothness of the IS-transformed integrand. Section~\ref{sec:ANOVA} further characterizes its ANOVA decomposition and Fourier spectrum. Section~\ref{sec:main} applies the theory to lattice rule integration and establishes the dimension-independent convergence of the proposed IS estimator. Section~\ref{sec:numer} presents numerical results confirming the effectiveness of our method. Section~\ref{sec:conclusion} concludes the paper. The appendices contain omitted proofs. 

%Section~\ref{sec:conclusion} concludes the paper.

\section{New proposal}\label{sec:method}
Estimating the integral \eqref{eq:integral} using IS estimators \eqref{eqn:muhat} requires finding suitable transport maps $T_j$. 
While previous studies focus on the choice of $T_j$ and the resulting smoothness of $w_j$, we take a different perspective: we first design $w_j$ with the desired smoothness, and then derive $T_j$ by solving the differential equation $w_j(u)=T_j'(u) \varphi\circ T_j(u)$, which gives
\begin{equation*}
    T_j(u)=\Phi^{-1}\Big(\int_{0}^u w_j(t)\rd t\Big).
\end{equation*}
We choose $w_j$ from a one-parameter family as follows. For $\theta\in (0, 1/2]$, let
\begin{equation}\label{eqn:wtheta}
w_\theta(u)=\begin{cases}
    (1-\theta)^{-1}\eta(u/\theta), &\text{ if } u\in (0,\theta/2], \\
    (1-\theta)^{-1}(1-\eta(1-u/\theta)),  &\text{ if } u\in (\theta/2,\theta), \\
    (1-\theta)^{-1}, &\text{ if } u\in [\theta,1/2], \\
    w_\theta(1-u), &\text{ if } u\in (1/2,1), \\
\end{cases}  
\end{equation}
where 
\begin{equation}\label{eq:tildeetap}
\eta(u)=2^{-3} u^{-2}\exp(2-u^{-1}).
\end{equation}
Notice that $\eta(u)$ is differentiable and monotonic on $(0,1/2]$, satisfying $\eta(0^+)=0$ and $\eta(1/2)=1/2$. Although other admissible choices for $\eta(u)$ exist, we adopt the function given in \eqref{eq:tildeetap} for its simplicity and computational convenience. 

By construction, $w_\theta(u)=w_\theta(1-u)$, $\sup_{u\in \itr} w_\theta(u)= (1-\theta)^{-1}$ and $\int_{0}^1 w_\theta(u)\rd u=1$.
Figure \ref{fig:wtheta} shows the case of $w_\theta(u)$ with $\theta = 0.2$, which decays to zero quickly as $u$ approaches $0$ or $1$.  

\begin{figure}[ht]
    \centering
    \includegraphics[width=0.6\linewidth]{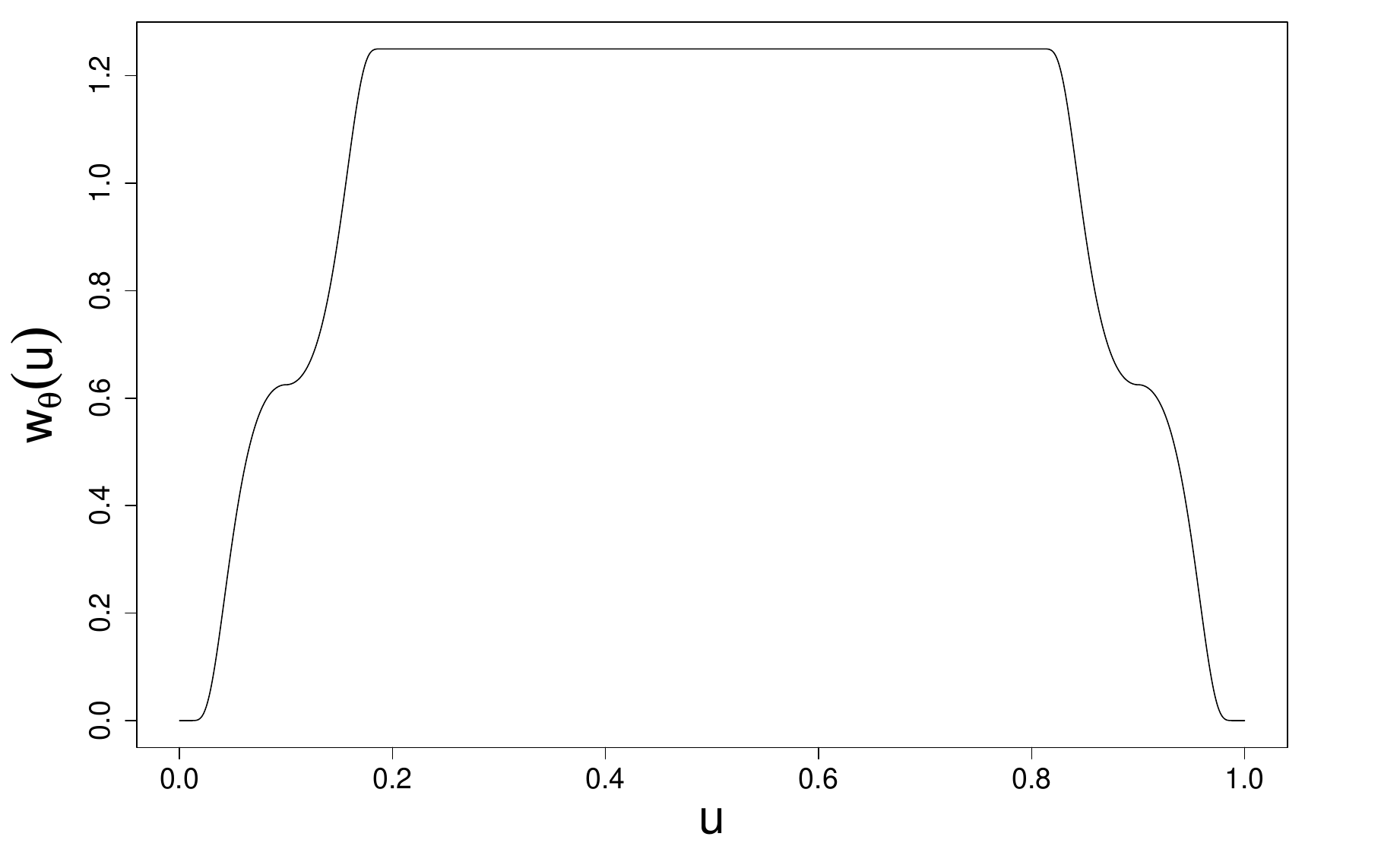}
    \caption{An illustration of $w_\theta(u)$ with $\theta =0.2$.}
    \label{fig:wtheta}
\end{figure}

In the remainder of this paper, we assume that $w_j(u)=w_{\theta_j}(u)$ for some $\theta_j\in(0,1/2]$. Accordingly,
\begin{equation}\label{eqn:Tj}
  T_j(u)=T_{\theta_j}(u)=\Phi^{-1}\left(\mathcal{I}_{\theta_j}(u)\right),
\end{equation}
where
$$
\mathcal{I}_{\theta}(u) :=\int_{0}^u w_{\theta}(t)= \begin{cases}
    \dfrac{\theta}{8(1-\theta)}\exp(2-\theta/u), &\text{ if } u\in (0,\theta/2], \\
    \dfrac{\theta}{8(1-\theta)}\exp(2-(1-u/\theta)^{-1})+\dfrac{u-\theta/2}{1-\theta}, &\text{ if } u\in (\theta/2,\theta), \\
    \dfrac{u-\theta/2}{1-\theta}, &\text{ if } u\in [\theta,1/2], \\
   1-\mathcal{I}_{\theta}(1-u) &\text{ if } u\in (1/2,1).
\end{cases} 
$$

Because $\bsT=(T_1,\dots,T_s)$ maps $\itr^s$ onto $\real^s$, 
\begin{equation*}\label{eqn:intfw}
    \int_{\itr^s}w(\bsu)f\circ \bsT(\bsu)\rd\bsu=\int_{\real^s} f(\bsx) \prod_{j=1}^s \varphi(x_j)\rd \bsx=\mu(f).
\end{equation*}
Hence, estimating $\mu(f)$ is equivalent to integrating a new integrand $$f^w(\bsu):=w(\bsu) f\circ \bsT(\bsu).$$
The weight function $w(\bsu)$ dampens the growth of the integrand $f^w(\bsu)$ near the boundary of $\itr^s$, so we call this method \textbf{boundary-damping IS}. 

Our goal is to show that with an appropriate choice of $\theta_j$, $f^w(\bsu)$ can be efficiently integrated even in high-dimensional settings. In particular, the mean squared error bounds are dimension-independent under conditions specified in Remark~\ref{rmk:tractability}.

\section{Preliminaries}\label{sec:back}
We first set up the notation. Let $1{:}s=\{1,\dots,s\}$.
For a vector $\bsx\in\real^s$ and a subset $v\subseteq 1{:}s$, $\Vert\bsx\Vert_2$ denotes the Euclidean norm, $\bsx_v$  denotes the subvector of $\bsx$ indexed by $v$, and $\bsx_{v^c}$ denotes the subvector indexed by $1{:}s\setminus v$. When $v=\{j\}$, we write $\bsx_v=x_j$ and $\bsx_{v^c}=\bsx_{-j}$ for simplicity.

Let $\natu$ be the set of positive integers, $\ints$ be the set of integers, and $\ints_*=\ints\setminus \{0\}$. For $\bsk\in \ints^s$, $\supp(\bsk)=\{j\in 1{:}s\mid k_j\neq 0\}$ denotes the support of $\bsk$. For a nonempty set $v\subseteq 1{:}s$, define the mixed derivative of $f$ by
$$\partial^v f(\bsx) = \left(\prod_{j\in v}\frac{\partial}{\partial x_j}\right)f(\bsx),$$
and define $\partial^\emptyset f(\bsx)=f(\bsx)$ by convention. 

We assume that the density $\varphi (x)$  is a strictly  positive, symmetric, light-tailed  function in the following.

\begin{assumption}\label{assump:rho}
Assume that $\varphi(x)>0$, $\varphi(x)=\varphi(-x)$, and for every $\varepsilon>0$, there exists a positive constant $c_\varepsilon$ such that
$\varphi(x)\geq c_\varepsilon \Phi(x)^{1+\varepsilon}$ for every $x\le 0$. 
\end{assumption}

Under Assumption~\ref{assump:rho},  the CDF $\Phi(x)$ and  its inverse $\Phi^{-1}(u)$ are almost everywhere differentiable and strictly increasing. It is easy to verify the Gaussian density $\varphi(x)=\exp(-x^2/2)/\sqrt{2\pi}$ satisfies Assumption~\ref{assump:rho} due to the inequality \cite{gordon1941values} 
$$\int_x^\infty \varphi(y)\rd y \leq \frac{1}{x}\varphi(x) \text{ for } x> 0.$$
The symmetry assumption can be dropped by imposing a proper condition on the right tail of the density $\varphi(x)$. 

All constants in this paper have an implicit dependency on $\varphi$ and we suppress it from the notation for simplicity.

\subsection{Function norms}

We will use the following function space norms:
\begin{align*}
    \Vert f \Vert_{L^q(\itr^s)}&=\Big(\int_{\itr^s} |f(\bsu)|^q \rd \bsu \Big)^{1/q},\\
    \Vert f \Vert_{W^{1,q}_{\mix}(\itr^s)}&=\Big(\sum_{v\subseteq 1{:}s} \Vert \partial^v  f \Vert^q_{L^q(\itr^s)}\Big)^{1/q},\\
    \Vert f \Vert_{L^q(\real^s,\varphi)}&=\Big(\int_{\real^s} |f(\bsx)|^q  \prod_{j=1}^s \varphi(x_j) \rd\bsx\Big)^{1/q},\\
    \Vert f \Vert_{W^{1,q}_{\mix}(\real^s,\varphi)}&=\Big(\sum_{v\subseteq 1{:}s} \Vert \partial^v  f \Vert^q_{L^q(\real^s,\varphi)}\Big)^{1/q}.
\end{align*}
We say that a function belongs to one of these spaces if the corresponding norm is finite. When $q=\infty$, the above norms all admit the usual generalization. When $s=1$, we write $W^{1,q}$ instead of $W^{1,q}_{\mix}$.

Observe that for constant functions $f(\bsx)=c$, all of the above norms reduce to $|c|$. More generally, if $f(\bsx)$ does not depend on $x_{v^c}$ for $v\subseteq 1{:}s$, we can treat $f$ either as a function over $\real^s$ or a function over $\real^{|v|}$ by removing redundant input variables $x_{v^c}$. Under either interpretation, the resulting norm is the same according to the definitions given above.

We will need the following Sobolev embedding lemma.

\begin{lemma}\label{lem:Sobolevimbedding}
    If $f\in W^{1,1}_{\mix}(\itr^s)$, then $\Vert f \Vert_{L^\infty(\itr^s)}\leq \Vert f \Vert_{W^{1,1}_{\mix}(\itr^s)}$.
\end{lemma}
\begin{proof}
    See Appendix~\ref{app:lem:Sobolevimbedding}.
\end{proof}

\subsection{ANOVA decomposition}

Our analysis of dimension structure relies on the ANOVA decomposition \cite{meandim}. We define the integration operator $I_j:L^1(\itr^s)\to L^1(\itr^s)$ by
$$I_j(f)(\bsu)=\int_{\itr} f(t_j,\bsu_{-j})\rd t_j \quad\text{for}\quad \bsu=(u_j,\bsu_{-j}) \in \itr^{s}.$$
Observe that $I_j(f)$ is constant in $u_j$ and that $I^2_j=I_j$. We further define the iterated integration operator $I_v=\prod_{j\in v}I_j$. By Fubini's theorem, the operators $I_j$ in $I_v$ commute. When $v=\emptyset$, $I_\emptyset$ is the identity operator over $L^1(\itr^s)$.  The ANOVA decomposition of $f\in L^1(\itr^s)$  is
\begin{equation}\label{eqn:anovadecomp}
    f=\sum_{v\subseteq 1{:}s} f_v \quad \text{for}\quad f_v=\left(I_{v^c} \prod_{j\in v} (I_\emptyset-I_j) \right)(f)=\sum_{v'\subseteq v} (-1)^{|v|-|v'|} I_{(v')^c}(f).
\end{equation}
It follows that for every $j\in v$, the ANOVA component $f_v$ satisfies
\begin{equation*}
  I_j (f_v)=\left((I_j-I_j^2)I_{ v^c} \prod_{j\in v\setminus\{j\}} (I_\emptyset-I_j) \right)(f)=0.  
\end{equation*}

Following \cite{kuo2010decompositions}, we generalize this framework to $L^1(\real^s,\varphi)$. Let $P_j:L^1(\real^s,\varphi)\to L^1(\real^s,\varphi)$ denote the integration operator
$$P_j(f)(\bsx)=\int_{\real} f(y_j,\bsx_{-j}) \varphi(y_j) \rd y_j \quad \text{for}\quad \bsx=(x_j,\bsx_{-j}) \in \real^{s}.$$
 We similarly define the iterated integration operator $P_v=\prod_{j\in v}P_j$ and the identity operator $P_\emptyset$. The ANOVA decomposition of $f\in L^1(\real^s,\varphi)$  is 
\begin{equation}\label{eqn:anovadecomp2}
 f=\sum_{v\subseteq 1{:}s} f_v \quad \text{for}\quad  f_v=\sum_{v'\subseteq v} (-1)^{|v|-|v'|} P_{ (v')^c}(f). 
\end{equation}
We also have $P_j (f_v)=0$ for every $j\in v$.

The next lemma shows that $f_v$ inherits the smoothness of $f$. 
\begin{lemma}\label{lem:anova}
Assume that the density $\varphi$  satisfies $\inf_{x\in K} \varphi(x)>0$ for every compact set $K\subseteq \real$. Let $f\in W^{1,q}_{\mix}(\real^s,\varphi)$ with $q\geq 1$, and let $v,w\subseteq 1{:}s$ satisfy $v\cap w=\emptyset$. Then the weak derivative \(\partial^{v}P_{w}(f)\) exists and satisfies
$$\partial^{v}P_{w}(f)=P_{w}(\partial^{v}f)\in L^q(\real^s,\varphi).$$
In particular, $f_v\in W^{1,q}_{\mix}(\real^s,\varphi)$ for all $v\subseteq 1{:}s$.
\end{lemma}
\begin{proof}
    See Appendix~\ref{app:lem:anova}.
\end{proof}

When $\varphi$ satisfies Assumption~\ref{assump:rho}, we have, for every $N,\varepsilon>0$,
$$\inf_{x\in [-N,N]} \varphi(x)\ge c_\varepsilon\Phi(-N)^{1+\varepsilon}>0,$$
so Lemma~\ref{lem:anova} applies.

\subsection{Lattice rules}\label{subsec:lattice}
In Section~\ref{sec:main}, we will choose the QMC point set in \eqref{eqn:muhat} to be the randomized lattice from \cite{KRITZER201996}. For a prime number $p$, a generating vector $\bsz=(z_1,\dots,z_s)\in \{1,\dots,p-1\}^s$, and a shift vector $\bsU=(U_1,\dots,U_s)\in \itr^s$, the shifted rank-1 lattice rule $Q_{p,\bsz,\bsU}=\{\bsu_i\}_{i=0}^{p-1}$ is given by
$$\bsu_i=\left(\left\{\frac{i\times z_1}{p}+U_1\right\},\dots,\left\{\frac{i\times z_s}{p}+U_s\right\}\right),$$
where \(\{x\}=x-\lfloor x\rfloor\) denotes the fractional part of $x$.

To state the error bound, we first introduce the weighted Sobolev space of dominating mixed smoothness.
For $f\in L^2(\itr^s)$ and $\bsk=(k_1,\dots,k_s)\in \ints^s$, we denote the Fourier coefficient of $f$ by
$$\widehat{f}(\bsk)=\int_{\itr^s} f(\bsu) \prod_{j=1}^s \exp(-\ii 2\pi k_ju_j)\rd \bsu .$$
Let $\alpha\geq0$, and let $\bsgamma=\{\gamma_j\mid j\in\natu\}$ be a non-increasing sequence of positive weights $1\ge \gamma_1\ge \gamma_2\ge \cdots >0$. We define the weighted Sobolev space norm
\begin{equation}\label{eqn:Korobovnorm}
    \Vert f\Vert_{s,\alpha,\bsgamma}=\left(\sum_{\bsk\in \ints^s} |r_{\alpha,\bsgamma}(\bsk) \widehat{f}(\bsk)|^2\right)^{1/2}, \quad\text{where}\quad r_{\alpha,\bsgamma}(\bsk)=\prod_{j\in \supp(\bsk)}\frac{|k_j|^{\alpha}}{\gamma_j}.
\end{equation}
We further use $\mathcal H_{s,\alpha,\bsgamma}$ to denote the set of functions $f\in L^2(\itr^s)$ with $\Vert f\Vert_{s,\alpha,\bsgamma}<\infty$. When $\alpha> 1/2$, the space $\mathcal H_{s,\alpha,\bsgamma}$ is often called the Korobov space \cite{dick2022lattice}.

The following theorem is a restatement of Theorem 11 in \cite{KRITZER201996}.

\begin{theorem}\label{thm:lattice}
Let \(\alpha>0\) and \(\bsgamma\in(0,1]^{\mathbb N}\) be a non-increasing sequence of positive weights. For an integer $n\ge 2$, let $p$ be drawn uniformly from
\begin{equation}\label{eqn:Pn}
   \mathcal P_n
:=
\left\{
p:\ p\text{ prime},\ 
\left\lceil \frac n2\right\rceil+1
\le p\le n
\right\}, 
\end{equation}
and let $\bsz$ be drawn uniformly from a set of good generating vectors \(\mathcal Z_{p,\alpha,\bsgamma}\) (see \cite[equation (12)]{KRITZER201996}). Further draw the shift $\bsU$ uniformly from $\itr^s$. Then for every $\lambda\in (0,\alpha)$ and $\delta\in (0,\min(1/\alpha,2))$, the QMC average $\widehat{\mu}(f)$ of $f\in \mathcal H_{s,\alpha,\bsgamma}$ over the lattice $Q_{p,\bsz,\bsU}$ satisfies
$$
\left(
\mathbb E
\left[
\left|
\widehat{\mu}(f)-\mu(f)
\right|^2
\right]
\right)^{1/2}\leq
\frac{\|f\|_{s,\alpha,\bsgamma}}{n^{\lambda+1/2-\delta\lambda/2}}\left(
\frac{c}{\alpha\delta}
\right)^{1/2}
\left(
12\prod_{j=1}^s(1+\gamma^{1/\lambda}_j\zeta(\alpha/\lambda))
\right)^{\lambda},$$
where \(c\) is an absolute constant, and $\zeta$ is the Riemann-zeta function.
\end{theorem}

Thus, for every \(\lambda<\alpha\) and $f\in  \mathcal H_{s,\alpha,\bsgamma}$ with $\|f\|_{s,\alpha,\bsgamma}\leq 1$, the algorithm achieves an error of order arbitrarily close to 
\(O(n^{-\alpha-1/2})\); if further $\sum_{j=1}^\infty \gamma^{1/\lambda}_j<\infty$, the implied constant is independent of the dimension.

\section{Smoothness of $f^w$}\label{sec:norm}
Recall that our estimator \eqref{eqn:muhat} is the QMC average of the IS-transformed integrand $f^w(\bsu)=w(\bsu) f\circ \bsT(\bsu)$. In this section, we characterize the smoothness of $f^w$ when $f\in W^{1,q}_{\mix}(\itr^s)$ with $q>1$. 

\begin{lemma}\label{lem:fwLq}
    If $f\in L^q(\real^s,\varphi)$ with $q\geq 1$, then
    \begin{equation*}
        \Vert f^w\Vert_{L^q(\itr^s)}\leq \left(\prod_{j=1}^s(1-\theta_j)^{-1}\right)^{\frac{q-1}{q}}\Vert f \Vert_{L^q(\real^s,\varphi)}\leq  2^{\frac{q-1}{q}s}\Vert f \Vert_{L^q(\real^s,\varphi)}.
    \end{equation*}
\end{lemma}
\begin{proof}
Because $T'_j(u)\geq 0$ and $\sup_{u\in\itr} w_j(u)= (1-\theta_j)^{-1}$,
\begin{align*}
    \Vert f^w\Vert^q_{L^q(\itr^s)}=&\int_{\itr^s} |f|^q\circ \bsT(\bsu)\prod_{j=1}^s w_j(u_j)^q\rd\bsu\\
    \leq &\left(\prod_{j=1}^s\sup_{u\in\itr} w_j(u)\right)^{q-1}\int_{\itr^s}|f|^q\circ \bsT(\bsu)\prod_{j=1}^s\varphi\circ T_j(u_j) T_j'(u_j)\rd\bsu\\
    \le &\left(\prod_{j=1}^s(1-\theta_j)^{-1}\right)^{q-1} \int_{\real^s} |f(\bsx)|^q \prod_{j=1}^s \varphi(x_j)\rd \bsx.
\end{align*}
Bounding $(1-\theta_j)^{-1}\leq 2$ completes the proof.
\end{proof}

The next two lemmas summarize useful facts about $\varphi\circ T_{\theta}$ and $w'_\theta$. Their proofs are given in Appendix~\ref{app:lem:rhoTbound}.

\begin{lemma}\label{lem:rhoTbound}
    Assume that $\theta\in (0,1/2]$ and $\varphi$ satisfies Assumption~\ref{assump:rho}. Then for every $\lambda>1$, $\varepsilon\in (0,\lambda-1)$, and $q\in [1,\infty]$, there exists a positive constant  $C_{\lambda,\varepsilon,q}$ such that
\begin{equation*}
    \left\|\frac{w_\theta^\lambda}{\varphi\circ T_\theta}\right\|_{L^q(\itr)}
    \le C_{\lambda,\varepsilon,q} \theta^{-1+1/q-\varepsilon}.  
\end{equation*}
\end{lemma}

\begin{lemma}\label{lem:dwjbound} 
Let $\theta\in (0,1/2]$. Then for every $\lambda\in (0,1)$ and $q\in [1,\infty]$, there exists a constant $C_{\lambda,q}$ such that
\begin{equation*}
    \left\| \frac{w'_\theta}{w_\theta^{\lambda}}\right\|_{L^q(\itr)} \leq C_{\lambda,q} \theta^{-1+1/q}. 
\end{equation*}
\end{lemma}

The next theorem bounds the  $W^{1,q}_{\mix}(\itr^s)$ norm of $f^w$ in terms of the $W^{1,q}_{\mix}(\real^s,\varphi)$ norm of $f$. 

\begin{theorem}\label{thm:Wqcase}
    Assume that $\varphi$ satisfies Assumption~\ref{assump:rho} and $f\in W^{1,q}_{\mix}(\real^s,\varphi)$ with $q>1$. Then for every $\varepsilon>0$ and $q'\in [1,q]$,
    $$\Vert f^w \Vert_{W^{1,q'}_{\mix}(\itr^s)}\leq C_{\varepsilon,q,q'}^s  \Vert f \Vert_{W^{1,q}_{\mix}(\real^s,\varphi)}\prod_{j=1}^s\theta_j^{-1-1/q+1/q'-\varepsilon},$$
    where $C_{\varepsilon,q,q'}$ is a constant depending on $\varepsilon,q$, and $q'$.
\end{theorem}
\begin{proof}
We assume $\varepsilon\in (0,1-1/q)$ without loss of generality. Set $q^*=q/(q-1)$, so that $1/q+1/q^*=1$.
A direct application of the chain rule for weak derivatives gives
\begin{equation*}\label{eqn:dfw}
    \partial^{v} f^w (\bsu) = \sum_{v'\subseteq v}(\partial^{v'} f) \circ \bsT(\bsu)\prod_{j\in 1{:}s\setminus v} w_j(u_j)\prod_{j\in v'} \frac{w_j(u_j)^{2}}{\varphi\circ T_j(u_j)}\prod_{j\in v\setminus v'} w'_j(u_j).
\end{equation*}
     After bounding $\sup_{u\in \itr} w_\theta(u)\leq 2$ for $\theta\in (0,1/2]$, we obtain
    \begin{equation}\label{eqn:partialfbound}
    \begin{aligned}
         |\partial^{v} f^w (\bsu)|\leq
        &2^{(s-|v|)/q^*}\sum_{v'\subseteq v}|(\partial^{v'} f) \circ \bsT(\bsu)| \prod_{j=1}^s w_j(u_j)^{1/q}\\
        &\times\prod_{j\in v'} \frac{w_j(u_j)^{2-1/q}}{\varphi\circ T_j(u_j)}\prod_{j\in v\setminus v'} \frac{|w'_j(u_j)|}{w_j(u_j)^{1/q}}.
    \end{aligned}
    \end{equation}
    Using the $L^\infty$ bounds from Lemmas~\ref{lem:rhoTbound} and \ref{lem:dwjbound}, we can find constants $C_{1}$ and $C_2$, depending on $\varepsilon$ and $q$, such that
   \begin{equation*}
       \begin{aligned}
           |\partial^{v} f^w (\bsu)|\leq &
        2^{(s-|v|)/q^*}\sum_{v'\subseteq v}|(\partial^{v'} f) \circ \bsT(\bsu)| \prod_{j=1}^s w_j(u_j)^{1/q} \prod_{j\in v'} C_1\theta^{-1-\varepsilon}_j\prod_{j\in v\setminus v'} C_2\theta^{-1}_j\\
        \leq & 2^{(s-|v|)/q^*}   \left(\sum_{v'\subseteq v}|(\partial^{v'} f) \circ \bsT(\bsu)|^q w(\bsu)\right)^{1/q}\left(\prod_{j\in v}C_{\varepsilon,q}\theta^{-1-\varepsilon}_j\right),
       \end{aligned}
   \end{equation*}
    where the last inequality uses Hölder's inequality and
    $C_{\varepsilon,q} =2^{1/q^*}\max(C_1,C_2)$.

    Raising both sides to the $q$-th power and integrating over $\itr^s$, we obtain
    $$\Vert \partial^{v} f^w\Vert_{L^q(\itr^s)}\leq 2^{(s-|v|)/q^*}  \Vert f \Vert_{W^{1,q}_{\mix}(\real^s,\varphi)}\prod_{j\in v}C_{\varepsilon,q}\theta^{-1-\varepsilon}_j .$$

    Next, integrating both sides of \eqref{eqn:partialfbound}  over $\itr^s$ and applying Hölder's inequality gives
    \begin{equation*}
         \Vert\partial^{v} f^w \Vert_{L^1(\itr^s)} \leq \\
    2^{(s-|v|)/q^*}\sum_{v'\subseteq v}\Vert \partial^{v'} f\Vert_{L^q(\real^s,\varphi)} \prod_{j\in v'} \left\|\frac{w_j^{2-1/q}}{\varphi\circ T_j}\right\|_{L^{q^*}(\itr)}\prod_{j\in v\setminus v'} \left\|\frac{w_j'}{w_j^{1/q}}\right\|_{L^{q^*}(\itr)}.
    \end{equation*}
By the $L^{q^*}$ bounds from Lemmas~\ref{lem:rhoTbound} and \ref{lem:dwjbound}, there exist $C_{1}'$ and $C_2'$, depending on $\varepsilon$ and $q$, such that
              \begin{align*}
        \Vert\partial^{v} f^w \Vert_{L^1(\itr^s)} 
    \leq & 2^{(s-|v|)/q^*}\sum_{v'\subseteq v}\Vert \partial^{v'} f\Vert_{L^q(\real^s,\varphi)} \prod_{j\in v'} C_{1}' \theta_j^{-1/q-\varepsilon} \prod_{j\in v\setminus v'} C_2' \theta_j^{-1/q}    \\
    \leq & 2^{(s-|v|)/q^*}\Vert f \Vert_{W^{1,q}_{\mix}(\real^s,\varphi)}\prod_{j\in v}C_{\varepsilon,q}'\theta^{-1/q-\varepsilon}_j ,
       \end{align*}
       where the last inequality uses Hölder's inequality and $C'_{\varepsilon,q} =2^{1/q^*}\max(C_1',C_2')$.
       
       Finally, applying the interpolation inequality \cite[Theorem 2.11]{adams2003sobolev} gives
       \begin{equation*}\label{eqn:fwinterpolate}
           \Vert\partial^{v} f^w \Vert_{L^{q'}(\itr^s)}\leq 2^{(s-|v|)/q^*} \Vert f \Vert_{W^{1,q}_{\mix}(\real^s,\varphi)}\prod_{j\in v}C_{\varepsilon,q,q'}\theta_j^{-1-1/q+1/q'-\varepsilon},
       \end{equation*}
       where $C_{\varepsilon,q,q'}=C^{1-\lambda}_{\varepsilon,q}(C'_{\varepsilon,q})^{\lambda}$ for $\lambda=q^*(1/q'-1/q)$.  Therefore,
       \begin{align*}
        \Vert f^w \Vert_{W^{1,q'}_{\mix}(\itr^s)}=&\left(\sum_{v\subseteq 1{:}s}\Vert\partial^{v} f^w \Vert^{q'}_{L^{q'}(\itr^s)}\right)^{1/q'}  \\
        \leq & \Vert f \Vert_{W^{1,q}_{\mix}(\real^s,\varphi)} \left(\sum_{v\subseteq 1{:}s}2^{(s-|v|)q'/q^*}\prod_{j\in v}C^{q'}_{\varepsilon,q,q'}\theta_j^{-q'(1+1/q-1/q'+\varepsilon)}\right)^{1/q'} \\
        = &  \Vert f \Vert_{W^{1,q}_{\mix}(\real^s,\varphi)} \prod_{j=1}^s\left(2^{q'/q^*}+C^{q'}_{\varepsilon,q,q'}\theta_j^{-q'(1+1/q-1/q'+\varepsilon)}\right)^{1/q'}.
       \end{align*}
       The claimed bound follows after increasing $C_{\varepsilon,q,q'}$ if necessary.
\end{proof}

\begin{corollary}\label{cor:Lqnorm}
    Under the assumptions of Theorem~\ref{thm:Wqcase},
    for every $\varepsilon>0$ and $q'\in [q,\infty]$,
    $$\Vert f^w \Vert_{L^{q'}(\itr^s)}\leq D_{\varepsilon,q,q'}^s  \Vert f \Vert_{W^{1,q}_{\mix}(\real^s,\varphi)}\prod_{j=1}^s\theta_j^{-1/q+1/q'-\varepsilon},$$
    where $D_{\varepsilon,q,q'}$ is a constant depending on $\varepsilon,q$, and $q'$.
\end{corollary}

\begin{proof}
    This follows immediately by interpolating between the $L^q$ bound of Lemma~\ref{lem:fwLq} and the $L^\infty$ bound of Lemma~\ref{lem:Sobolevimbedding}.
\end{proof}

Finally, we show that $f^w$ vanishes on the boundary of $\itr^s$.

\begin{corollary}\label{cor:periodic}
    Under the assumptions of Theorem~\ref{thm:Wqcase},
    after modifying $f^w$ on a null set if necessary, $f^w(\bsu)\to0$ as $\mathrm{dist}(\bsu,[0,1]^s\setminus \itr^s)\to 0$,
    where 
    $$\mathrm{dist}(\bsu,[0,1]^s\setminus \itr^s)=\min_{j\in 1{:}s} \min(u_j,1-u_j) \quad \text{for}\quad \bsu\in \itr^s.$$
\end{corollary}
\begin{proof}
Let $\delta\in (0,1-1/q)$ and
\[
g(\bsu)=f^w(\bsu)w(\bsu)^{-\delta}=f\circ\bsT(\bsu)w(\bsu)^{1-\delta}.
\]
For every \(v\subseteq 1{:}s\), differentiating \(g\) as in Lemma~\ref{eqn:dfw} gives
\begin{align*}
\partial^v g(\bsu)
&=
\sum_{v'\subseteq v}
(\partial^{v'} f)\circ \bsT(\bsu)
\prod_{j\in 1{:}s\setminus v} w_j(u_j)^{1-\delta}
\prod_{j\in v'} \frac{w_j(u_j)^{2-\delta}}{\varphi\circ T_j(u_j)}
\prod_{j\in v\setminus v'} \frac{(1-\delta)w'_j(u_j)}{w_j(u_j)^{\delta}} .
\end{align*}
As in \eqref{eqn:partialfbound},
    \begin{equation*}
    \begin{aligned}
        |\partial^v g(\bsu)|\leq
        &2^{(s-|v|)(1-1/q-\delta)}\sum_{v'\subseteq v}|(\partial^{v'} f) \circ \bsT(\bsu)| \prod_{j=1}^s w_j(u_j)^{1/q}\\
        &\times\prod_{j\in v'} \frac{w_j(u_j)^{2-1/q-\delta}}{\varphi\circ T_j(u_j)}\prod_{j\in v\setminus v'} \frac{|w'_j(u_j)|}{w_j(u_j)^{1/q+\delta}}.
    \end{aligned}
    \end{equation*}
By the $L^\infty$ bounds from Lemmas~\ref{lem:rhoTbound} and \ref{lem:dwjbound}, $\Vert \partial^v g\Vert_{L^q(\itr^s)}<\infty$ given $f\in W^{1,q}_{\mix}(\real^s,\varphi)$. Thus \(g\in W^{1,1}_{\mix}(\itr^s)\). By  Lemma~\ref{lem:Sobolevimbedding}, after modifying \(g\) on a null set if necessary, there exists a finite constant \(M\) such that $\sup_{\bsu\in\itr^s}|g(\bsu)|\leq M$.
Therefore,
\[
|f^w(\bsu)|
=
|g(\bsu)|w(\bsu)^\delta
\le M w(\bsu)^\delta.
\]
Now let $\mathrm{dist}(\bsu,[0,1]^s\setminus \itr^s)\to 0$.
Then \(u_j\to0\) or \(u_j\to1\) for at least one \(j\). Since \(w_j\) is continuous on \(\itr\) and \(w_j(0^+)=w_j(1^-)=0\), it follows that \(w_j(u_j)\to0\), and hence \(f^w(\bsu)\to0\).
\end{proof}

\section{Dimension structure of $f^w$}\label{sec:ANOVA}

We have shown that $f^w\in W^{1,1}_{\mix}(\itr^s)\subseteq L^\infty(\itr^s)$ if $f\in W^{1,q}_{\mix}(\real^s,\varphi)$ with $q>1$. However, the $W^{1,1}_{\mix}(\itr^s)$ norm grows exponentially with $s$. To derive dimension-independent error bounds, we decompose $f^w$ into its ANOVA components and characterize their Fourier specturms.

\begin{lemma}\label{lem:fwk}
Let $f\in L^1(\real^s,\varphi)$ with ANOVA components $f_v$ as in \eqref{eqn:anovadecomp2}. Then the ANOVA decomposition of $f^w$ is
\begin{equation*}\label{eqn:fwv}
  f^w=\sum_{v\subseteq 1{:}s} f^w_v \quad \text{for} \quad   f^w_v=\sum_{v'\subseteq v}(f_{v'}\circ\bsT)\prod_{j\in v'} w_j\prod_{j\in v\setminus v'} (w_j-1),
\end{equation*}
where $w_j$ is evaluated at $u_j$, and $1$ denotes the constant function equal to one.
\end{lemma}
\begin{proof}
    Writing $f$ in its ANOVA decomposition gives 
    $$f^w=\sum_{v'\subseteq 1{:}s} (f_{v'}\circ\bsT) \prod_{j\in 1{:}s} w_j.$$
    For every $j\in v'$,
    \begin{equation}\label{eqn:Ijequal0}
        \begin{aligned}
            I_{j}\left((f_{v'}\circ\bsT) w_j\right)&=\int_{\itr} f_{v'}\circ\bsT(\bsu)\varphi\circ T_j(u_j) T_j'(u_j) \rd u_j\\
        &=\int_\real f_{v'}(\bsx) \varphi(x_j)\rd x_j=P_{j}\left(f_{v'}\right)=0.
        \end{aligned}
    \end{equation}
    Hence, 
    $$I_{1{:}s\setminus v} (f^w)=\sum_{v'\subseteq v} (f_{v'}\circ\bsT)\prod_{j\in v}w_j \prod_{j\in 1{:}s\setminus v} I_j(w_j)=\sum_{v'\subseteq v} f_{v'}\circ\bsT(\bsu)\prod_{j\in v} w_j,$$
    where the second equality is due to $I_j(w_j)=1$. Finally, by \eqref{eqn:anovadecomp},
    \begin{align*}
        f^w_v=&\left(\prod_{j\in v} (I_\emptyset-I_j)\right)\left(I_{1{:}s\setminus v} (f^w)\right)\\
        =&\sum_{v'\subseteq v} \left(\prod_{j\in v'} (I_\emptyset-I_j)\right)\left((f_{v'}\circ\bsT)\prod_{j\in v'} w_j\right)\times\prod_{j\in v\setminus v'}(I_\emptyset-I_j)(w_j)\\
        =&\sum_{v'\subseteq v}(f_{v'}\circ\bsT)\prod_{j\in v'} w_j\prod_{j\in v\setminus v'} (w_j-1),
    \end{align*}
    where the last equality uses \eqref{eqn:Ijequal0} and $I_j(w_j)=1$.
\end{proof}

We will need the following bounds on \(w_j-1\).

\begin{lemma}\label{lem:wthetanorm}
For \(\theta\in(0,1/2]\),
\[
\|w_\theta-1\|_{L^2(\itr)}\le 2\theta^{1/2}
\quad\text{and}\quad
\|w_\theta-1\|_{W^{1,1}(\itr)}\le 6.
\]
\end{lemma}

\begin{proof}
We first bound \(\|w_\theta-1\|_{L^1(\itr)}\). By the triangle inequality,
\[
\|w_\theta-1\|_{L^1(\itr)}
\le
\|w_\theta-(1-\theta)^{-1}\|_{L^1(\itr)}
+\|(1-\theta)^{-1}-1\|_{L^1(\itr)}.
\]
Since \(w_\theta\le(1-\theta)^{-1}\) on \(\itr\),
\[
\|w_\theta-(1-\theta)^{-1}\|_{L^1(\itr)}
=
\int_\itr\bigl((1-\theta)^{-1}-w_\theta(u)\bigr)\,du
=
(1-\theta)^{-1}-1.
\]
Similarly, \(\|(1-\theta)^{-1}-1\|_{L^1(\itr)}=(1-\theta)^{-1}-1\). Hence, for \(\theta\in(0,1/2]\),
\[
\|w_\theta-1\|_{L^1(\itr)}
\le 2\bigl((1-\theta)^{-1}-1\bigr)
\le 4\theta.
\]
The \(L^2\) bound follows by interpolating this \(L^1\) bound with the \(L^\infty\) bound
\[
\|w_\theta-1\|_{L^\infty(\itr)}\le\max\bigl((1-\theta)^{-1}-1,1\bigr)\le 1.
\]

Next, we bound the \(W^{1,1}(\itr)\) norm. By symmetry and the fact that \(w_\theta\) increases monotonically from \(0\) to \((1-\theta)^{-1}\) on \((0,1/2)\),
\[
\|w'_\theta\|_{L^1(\itr)}
=
2\int_0^{1/2} |w'_\theta(u)|\,du
=
2(1-\theta)^{-1}
\le 4.
\]
Therefore,
\[
\|w_\theta-1\|_{W^{1,1}(\itr)}
=
\|w_\theta-1\|_{L^1(\itr)}+\|w'_\theta\|_{L^1(\itr)}
\le 4\theta+4
\le 6.
\]
\end{proof}

Recall that $\ints_*$ is the set of nonzero integers. The next lemma bounds the Fourier coefficients supported on $\ints_*^s$.
\begin{lemma}
  Let $g\in W^{1,1}_{\mix}(\itr^s)$. Then for every $\alpha\in (0,1/2)$ and $\delta\in (0,2-4\alpha)$,
  \begin{equation}\label{eqn:alphanorm}
      \sum_{\bsk\in \ints_*^s}|\widehat{g}(\bsk)|^2\prod_{j=1}^s |k_j|^{2\alpha}\leq C_{\alpha,\delta}^{s} \|g\|^{2-4\alpha-\delta}_{L^2(\itr^{s})} \|\partial^{1{:}s} g\|^{4\alpha+\delta}_{L^1(\itr^s)},
  \end{equation}
    where $C_{\alpha,\delta}$ is a constant depending on $\alpha$ and $\delta$.
\end{lemma}
\begin{proof}
Set  \(\lambda=2\alpha+\delta/2\in (0,1) \). By Hölder's inequality,
    \begin{equation}\label{eqn:Holder}
        \begin{aligned}
            \sum_{\bsk\in\ints_*^s}|\widehat{g}(\bsk)|^2\prod_{j=1}^s|k_j|^{2\alpha}
        \le
        \Bigl(\sum_{\bsk\in\ints_*^s}|\widehat{g}(\bsk)|^2\Bigr)^{1-\lambda}
        \Bigl(\sum_{\bsk\in\ints_*^s}|\widehat{g}(\bsk)|^2
        \prod_{j=1}^s|k_j|^{2\alpha/\lambda}\Bigr)^{\lambda}.
        \end{aligned}
    \end{equation}
The first factor is bounded by the $L^2$ norm using Parseval's identity. It remains to bound the second factor.

For every \(\bsk\in\ints_*^s\), since \(\exp(-\mathrm{i}2\pi k_j u_j)-1\) vanishes at both \(u_j=0\) and \(u_j=1\), integration by parts gives
\[
\widehat{g}(\bsk)
=
\frac{1}{\prod_{j=1}^s(\mathrm{i}2\pi k_j)}
\int_{\itr^s}
\partial^{1{:}s} g(\bsu)
\prod_{j=1}^s\bigl(\exp(-\mathrm{i}2\pi k_j u_j)-1\bigr)\rd\bsu.
\]
Taking absolute values and using \(|\exp(-\mathrm{i}2\pi k_j u_j)-1|\le 2\),
\begin{equation*}
    |\widehat{g}(\bsk)|
    \le
    \frac{1}{\prod_{j=1}^s(\pi|k_j|)}
    \|\partial^{1{:}s} g\|_{L^1(\itr^s)}.
\end{equation*}
Hence,
\begin{align*}
    \sum_{\bsk\in\ints_*^s}|\widehat{g}(\bsk)|^2
    \prod_{j=1}^s|k_j|^{2\alpha/\lambda}
    &\le
    \|\partial^{1{:}s} g\|^2_{L^1(\itr^s)}
    \sum_{\bsk\in\ints_*^s}
    \prod_{j=1}^s
    \frac{|k_j|^{2\alpha/\lambda}}{(\pi|k_j|)^2}\\
    &=\|\partial^{1{:}s} g\|^2_{L^1(\itr^s)}
   \left(\frac{2\zeta(2-2\alpha/\lambda)}{\pi^2}\right)^s,
\end{align*}
where \(\zeta\) is the Riemann-zeta function.

Substituting the bound into \eqref{eqn:Holder} yields
\[
\sum_{\bsk\in\ints_*^s}|\widehat{g}(\bsk)|^2\prod_{j=1}^s|k_j|^{2\alpha}
\le
\left(\frac{2\zeta(2-2\alpha/\lambda)}{\pi^2}\right)^{\lambda s}
\|g\|^{2(1-\lambda)}_{L^2(\itr^s)}
\|\partial^{1{:}s} g\|^{2\lambda}_{L^1(\itr^s)}.
\]
The proof is complete after substituting \(\lambda=2\alpha+\delta/2\).
\end{proof}

\begin{theorem}\label{thm:fkbound}
    Assume that $q>1$, $\varphi$ satisfies Assumption~\ref{assump:rho}, and \(f\in W^{1,q}_{\mix}(\real^s,\varphi)\) with ANOVA components $f_v$. Then for every $\varepsilon>0$ and \(\alpha\in (0,1/2)\),
\begin{equation*}
        \sum_{\bsk\in \ints^s}|\widehat{f_v^w}(\bsk)|^2\prod_{j=1}^s (1+k^2_j)^{\alpha}
        \leq C_{\varepsilon,q,\alpha}^{|v|} \sum_{v'\subseteq v}\Vert f_{v'} \Vert^2_{W^{1,q}_{\mix}(\real^s,\varphi)}\prod_{j\in v'} \theta^{-\nu_{q,\alpha}-\varepsilon}_j\prod_{j\in v\setminus v'} \theta^{1-2\alpha-\varepsilon}_j,
    \end{equation*}
    where $C_{\varepsilon,q,\alpha}$ is a constant depending on $\varepsilon,q$, and $\alpha$, and
            $$
        \nu_{q,\alpha}=\begin{cases}
            2/q+2\alpha-1, &\text{ if } q\in(1,2],\\
            4\alpha/q, &\text{ if } q\in(2,\infty).
        \end{cases}
        $$
\end{theorem}
\begin{proof}

By Lemma~\ref{lem:fwk}, the ANOVA component \(f_v^w\) can be written as
\[
f_v^w = \sum_{v'\subseteq v} g_{v'}\prod_{j\in v\setminus v' }h_{j}, \quad \text{where} \quad g_{v'}  = (f_{v'}\circ\bsT)\prod_{j\in v'} w_j, \ h_{j} =  w_j-1.
\]

Recall that  $\supp(\bsk)=\{j\in 1{:}s\mid k_j\neq 0\}$.
By \eqref{eqn:Ijequal0}, $\widehat{g_{v'} }(\bsk)=0$ unless $\supp(\bsk)=v'$. Similarly, $I_j(w_j)=1$ implies that $\widehat{h_j}(0)=0$. Consequently, the product $F_{v'}=g_{v'}\prod_{j\in v\setminus v' }h_{j}$ satisfies
\begin{equation}\label{eq:F-factor}
\begin{aligned}
\sum_{\bsk\in\ints^{s}}|\widehat{F_{v'}}(\bsk)|^2\prod_{j=1}^s(1+k_j^2)^\alpha 
=
\left(\sum_{\bsk_{v'}\in \ints_*^{|v'|}}|\widehat{g_{v'}}(\bsk_{v'})|^2\prod_{j\in v'}(1+k_j^2)^\alpha\right)
\left(\prod_{j\in v\setminus v'}S_j\right),
\end{aligned}
\end{equation}
where
$$S_j=\sum_{k\in\ints_*}|\widehat{h_j}(k)|^2(1+k^2)^\alpha.$$
By Lemma~\ref{lem:wthetanorm}, $\|h_j\|_{L^2(\itr)}\leq 2\theta_j^{1/2}$ and $\|h'_j\|_{L^1(\itr)}\leq \|h_j\|_{W^{1,1}(\itr)} \leq 6$. For all sufficiently small $\varepsilon>0$, \eqref{eqn:alphanorm} with $\delta=2\varepsilon$ then gives
\begin{equation}\label{eq:B-bound}
    S_j\leq 2^\alpha \sum_{k\in\ints_*}|\widehat{h_j}(k)|^2|k|^{2\alpha}\leq C_{\varepsilon,\alpha} \theta_j^{1-2\alpha-\varepsilon},
\end{equation}
where $C_{\varepsilon,\alpha}$ depends on $\varepsilon$ and $\alpha$.

To bound the first factor in \eqref{eq:F-factor}, observe that the function \(g_{v'}\) is the IS-transform of \(f_{v'}\) restricted to the variables in \(v'\). If $q> 2$, Lemma~\ref{lem:fwLq} implies that
\begin{equation*}
    \|g_{v'} \|_{L^2(\itr^{|v'|})}
\le \|g_{v'} \|_{L^q(\itr^{|v'|})}\leq 2^{(q-1)|v'|/q}\|f_{v'}\|_{L^q(\real^{|v'|},\varphi)}.
\end{equation*}
If $q\in (1,2]$, Corollary~\ref{cor:Lqnorm} implies that
\begin{equation*}
    \|g_{v'} \|_{L^2(\itr^{|v'|})}
\leq D_{\varepsilon,q,2}^{|v'|}  \|f_{v'}\|_{W^{1,q}_{\mix}(\real^{|v'|},\varphi)}\prod_{j\in v'}\theta_j^{-1/q+1/2-\varepsilon}.
\end{equation*}
Moreover, by Theorem~\ref{thm:Wqcase} with \(q'=1\),
\begin{equation*}
    \|g_{v'} \|_{W^{1,1}_{\mix}(\itr^{|v'|})}
\le C_{\varepsilon,q,1}^{|v'|}\|f_{v'}\|_{W^{1,q}_{\mix}(\real^{|v'|},\varphi)}
\prod_{j\in v'}\theta_j^{-1/q-\varepsilon}.
\end{equation*}
Plugging the above bounds into \eqref{eqn:alphanorm} with a sufficiently small $\delta>0$, we obtain
\begin{equation}\label{eq:A-bound}
\begin{aligned}
    \sum_{\bsk_{v'}\in \ints_*^{|v'|}}|\widehat{g_{v'} }(\bsk_{v'})|^2\prod_{j\in v'}(1+k_j^2)^\alpha
    \leq & 2^{\alpha|v'|} \sum_{\bsk_{v'}\in \ints_*^{|v'|}}|\widehat{g_{v'} }(\bsk_{v'})|^2\prod_{j\in v'}|k_j|^{2\alpha}\\
    \leq &C_{\varepsilon,q,\alpha}^{|v'|}\|f_{v'}\|^2_{W^{1,q}_{\mix}(\real^{s},\varphi)}
\prod_{j\in v'}\theta_j^{-\nu_{q,\alpha}-\varepsilon},
\end{aligned}
\end{equation}
where $C_{\varepsilon,q,\alpha}$ depends on $\varepsilon,q,$ and $\alpha$, and we have used $\|f_{v'}\|_{W^{1,q}_{\mix}(\real^{|v'|},\varphi)}=\|f_{v'}\|_{W^{1,q}_{\mix}(\real^{s},\varphi)}$, with the former $f_{v'}$ viewed as a function of \(\bsx_{v'}\in\real^{|v'|}\) and the latter as a function of \(\bsx\in\real^{s}\). 

Finally, substituting \eqref{eq:B-bound} and \eqref{eq:A-bound} into \eqref{eq:F-factor} yields
\begin{equation}\label{eqn:Fv}
    \sum_{\bsk\in\ints^s}|\widehat{F_{v'}}(\bsk)|^2\prod_{j=1}^s(1+k_j^2)^\alpha
\le C_{\varepsilon,q,\alpha}^{|v'|}C^{|v\setminus v'|}_{\varepsilon,\alpha}\|f_{v'}\|^2_{W^{1,q}_{\mix}(\real^{s},\varphi)}
\prod_{j\in v'}\theta_j^{-\nu_{q,\alpha}-\varepsilon}
\prod_{j\in v\setminus v'}\theta_j^{1-2\alpha-\varepsilon}.
\end{equation}
Since \(f_v^w=\sum_{v'\subseteq v}F_{v'}\),
$$|\widehat{f_v^w}(\bsk)|^2\leq \left(\sum_{v'\subseteq v}|\widehat{F_{v'}}(\bsk)|\right)^2\leq 2^{|v|}\sum_{v'\subseteq v}|\widehat{F_{v'}}(\bsk)|^2.$$
Hence,
\[
\sum_{\bsk\in \ints^s}|\widehat{f_v^w}(\bsk)|^2\prod_{j=1}^s(1+k_j^2)^\alpha
\le
2^{|v|}\sum_{v'\subseteq v}
\sum_{\bsk\in\ints^s}|\widehat{F_{v'}}(\bsk)|^2\prod_{j=1}^s(1+k_j^2)^\alpha.
\]
The proof is complete once we plug in the bound \eqref{eqn:Fv}.
\end{proof}

\section{Application to lattice rules}\label{sec:main}

In this section, we derive dimension-independent error bounds for the randomized lattice rule of Section~\ref{subsec:lattice} applied to the IS-transformed integrand $f^w$. 

We first bound the weighted Sobolev norm $\|f^w\|_{s,\alpha,\bsgamma}$ introduced in \eqref{eqn:Korobovnorm}. To simplify the notation, we let
\begin{equation*}
     \Vert f\Vert_{q}=\sup_{ v\subseteq 1{:}s}\Vert f_v \Vert_{W^{1,q}_{\mix}(\real^s,\varphi)}.
\end{equation*}

\begin{theorem}\label{thm:Korobovnorm}
  Assume that $q>1$, $\varphi$ satisfies Assumption~\ref{assump:rho}, and \(f\in W^{1,q}_{\mix}(\real^s,\varphi)\) with ANOVA components $f_v$. If there exists a sequence of positive weights $\{\Gamma_j\}_{j\geq 1}$ such that
    $$\Vert f_v \Vert_{W^{1,q}_{\mix}(\real^s,\varphi)}\leq \Vert f\Vert_{q} \prod_{j\in v}\Gamma_j \quad\text{for all}\quad v\subseteq 1{:}s,$$
    then for every $\varepsilon>0$, $\alpha\in (0,1/2)$, and weights $\bsgamma=\{\gamma_j\mid j\in\natu\}\subseteq (0,1]^\natu$,
    $$\|f^w\|^2_{s,\alpha,\bsgamma}\leq \|f\|^2_q \prod_{j=1}^s \left[1+C_{\varepsilon,q,\alpha}\gamma_j^{-2} \left(\Gamma_j^2\theta^{-\nu_{q,\alpha}-\varepsilon}_j+\theta^{1-2\alpha-\varepsilon}_j\right)\right],$$
    where $C_{\varepsilon,q,\alpha}$ and $\nu_{q,\alpha}$ are from Theorem~\ref{thm:fkbound}.
\end{theorem}
\begin{proof}
    Since $f^w=\sum_{v\subseteq 1{:}s}f^w_v$ and $I_j(f^w_v)=0$ for all $j\in v$,  $\widehat{f^w}(\bsk)=\widehat{f_v^w}(\bsk)$ when $\supp(\bsk)=v$. Therefore,
    $$\|f^w\|_{s,\alpha,\bsgamma}^2=\sum_{\bsk\in \ints^s}  |r_{\alpha,\bsgamma}(\bsk)\widehat{f^w}(\bsk)|^2=\sum_{v\subseteq 1{:}s}\sum_{\substack{\bsk\in \ints^s\\\supp(\bsk)=v}}|r_{\alpha,\bsgamma}(\bsk)|^2|\widehat{f_v^w}(\bsk)|^2.$$
    When $\supp(\bsk)=v$, 
    $$|r_{\alpha,\bsgamma}(\bsk)|^2=\prod_{j\in v}\frac{|k_j|^{2\alpha}}{\gamma_j^2}
    \leq\left(\prod_{j\in v}\gamma_j^{-2} \right)\prod_{j=1}^s (1+k_j^2)^{\alpha}.$$
    By Theorem~\ref{thm:fkbound}, we conclude that
    \begin{align*}
        \|f^w\|_{s,\alpha,\bsgamma}^2
        \leq &\|f\|_q^2 \sum_{v\subseteq 1{:}s}\left(\prod_{j\in v}\gamma_j^{-2} \right) C_{\varepsilon,q,\alpha}^{|v|} \sum_{v'\subseteq v}\prod_{j\in v'}\Gamma_j^2 \theta^{-\nu_{q,\alpha}-\varepsilon}_j\prod_{j\in v\setminus v'} \theta^{1-2\alpha-\varepsilon}_j\\
        \leq & \|f\|_q^2 \sum_{v\subseteq 1{:}s} C_{\varepsilon,q,\alpha}^{|v|}\left[\prod_{j\in v}\gamma_j^{-2} \left(\Gamma_j^2\theta^{-\nu_{q,\alpha}-\varepsilon}_j+\theta^{1-2\alpha-\varepsilon}_j\right)\right]\\
        =& \|f\|_q^2 \prod_{j=1}^s \left[1+C_{\varepsilon,q,\alpha}\gamma_j^{-2} \left(\Gamma_j^2\theta^{-\nu_{q,\alpha}-\varepsilon}_j+\theta^{1-2\alpha-\varepsilon}_j\right)\right].
    \end{align*}
\end{proof}

    \begin{corollary}\label{thm:qmcvar}
Under the assumptions of Theorem~\ref{thm:Korobovnorm}, fix  $\alpha\in (0,1/2)$ and $\varepsilon\in (0,4\alpha)$. Let
$$C_{\Upsilon}=\sup_{1\leq j\leq s}\Upsilon_j, \quad \text{where}\quad\Upsilon_j=\Gamma_j^2\theta_j^{-\nu_{q,\alpha}-\varepsilon}+\theta^{1-2\alpha-\varepsilon}_j,$$
and set   $$\gamma_j=\left(C_{\Upsilon}^{-1}\Upsilon_j\right)^{\frac{\alpha}{1+2\alpha}}.$$
    Then for every integer $n\geq 2$, the  randomized lattice $Q_{p,\bsz,\bsU}=\{\bsu_i\}_{i=0}^{p-1}$  defined in Theorem~\ref{thm:lattice} satisfies
    \begin{equation}\label{eqn:errorbound}
        \e \left[ \left|\frac{1}{p}\sum_{i=0}^{p-1}f^w(\bsu_i)-\mu(f)\right|^2 \right]
            \leq 
\frac{C'_{\varepsilon,q,\alpha}\|f\|^2_q }{n^{1+2\alpha-\varepsilon}}\prod_{j=1}^s \left(1+C'_{\varepsilon,q,\alpha} C_{\Upsilon,\alpha}\Upsilon_j^{\frac{1}{1+2\alpha}}\right)^{1+2\alpha},
    \end{equation}
        where $C'_{\varepsilon,q,\alpha}$ is a constant depending on $\varepsilon,q,$ and $\alpha$, and $$C_{\Upsilon,\alpha}=\max\Bigl(C_{\Upsilon}^{\frac{2\alpha}{1+2\alpha}},C_{\Upsilon}^{-\frac{1}{1+2\alpha}}\Bigr).$$
\end{corollary}
    \begin{proof}
    Applying Theorem~\ref{thm:lattice} with $\lambda=\alpha-\varepsilon/4$ and $\delta=\varepsilon$ gives
    $$\e \left[ \left|\frac{1}{p}\sum_{i=0}^{p-1}f^w(\bsu_i)-\mu(f)\right|^2 \right]\leq\frac{\|f^w\|^2_{s,\alpha,\bsgamma}\left(
\frac{c}{\alpha\varepsilon}
\right)}{n^{(1-\varepsilon/2)(2\alpha+1)}}
12^{2\alpha}\prod_{j=1}^s\left[1+\gamma_j^{1/\alpha}\zeta\left(\frac{\alpha}{\alpha-\varepsilon/4}\right)\right]^{2\alpha}.$$
    Bounding 
    $\|f^w\|_{s,\alpha,\bsgamma}$ by Theorem~\ref{thm:Korobovnorm} yields the claimed bound.
    \end{proof}
    
\begin{remark}\label{rem:RQMCrate}
When \(\varphi\) is the standard Gaussian density and \(f\) satisfies
\[
\sup_{v\subseteq 1{:}s} |\partial^v f(\bsx)|=O\bigl(\exp(M\|\bsx\|_2^2)\bigr)
\quad\text{as}\quad \|\bsx\|_2\to\infty
\]
for some \(M\in(0,1/2)\), plain randomized QMC without IS has a mean squared error rate of \(O(n^{-2+4M+\varepsilon})\) \cite{he:2023}. As \(M\) approaches \(1/2\), this rate degenerates to zero. In contrast, our boundary-damping IS improves the rate to \(O(n^{-2+\varepsilon})\) for every \(M\in(0,1/2)\), since \(f\in W^{1,q}_{\mix}(\real^s,\varphi)\) for all \(q\in(1,1/(2M))\).
\end{remark}

\begin{remark}\label{rmk:tractability}
   When  $\sum_{j=1}^\infty \Upsilon^{1/(1+2\alpha)}_j<\infty$, the mean squared error \eqref{eqn:errorbound} is further uniformly bounded in $s$ because $C_\Upsilon\leq \sup_{j\geq 1} \Upsilon_j<\infty$ and
   $$\prod_{j=1}^s \left(1+C'_{\varepsilon,q,\alpha} C_{\Upsilon,\alpha}\Upsilon_j^{\frac{1}{1+2\alpha}}\right)^{1+2\alpha}\leq \exp\left(C'_{\varepsilon,q,\alpha} C_{\Upsilon,\alpha}\sum_{j=1}^\infty \Upsilon_j^{\frac{1}{1+2\alpha}}\right)<\infty.$$

For instance, assume that $\Gamma_j=O(j^{-\rho})$ for $\rho>1$. We balance the two terms $\Gamma_j^2 \theta^{-\nu_{\alpha,q}-\varepsilon}_j$ and $\theta^{1-2\alpha-\varepsilon}_j$ in $\Upsilon_j$ by setting
\begin{equation}\label{eqn:thetachoice}
 \theta_j=\begin{cases}
    \theta_0 j^{-\rho q}, &\text{ if } q\in (1,2], \\
    \theta_0 j^{-2\rho/(4\alpha/q+1-2\alpha)}, &\text{ if } q\in (2,\infty),
\end{cases}   
\end{equation}
with a constant $\theta_0\in (0,1/2]$. 

For \(q\in (1,2]\), the above choice gives
\[
\Upsilon_j=O\!\left(j^{-\rho q(1-2\alpha)+\varepsilon'}\right),
\]
where $\varepsilon'=\rho q\varepsilon>0$ is arbitrarily small.
To ensure \(\sum_{j=1}^\infty \Upsilon_j^{1/(1+2\alpha)}<\infty\), we require that
\[
\frac{\rho q(1-2\alpha)}{1+2\alpha}>1.
\]
This leads to the condition
\begin{equation}\label{eq:alpha bound}
0<\alpha<\alpha^*(q,\rho):=\frac{\rho q-1}{2(\rho q+1)}.
\end{equation}

For \(q>2\), we have
\[
\Upsilon_j=O\!\left(j^{-\frac{2\rho(1-2\alpha)}{4\alpha/q+1-2\alpha}+\varepsilon''}\right),
\]
where \(\varepsilon''=2\rho\varepsilon/(4\alpha/q+1-2\alpha)>0\) is arbitrarily small. To ensure \(\sum_{j=1}^\infty \Upsilon^{1/(1+2\alpha)}_j<\infty\), we need
\[
\frac{2\rho(1-2\alpha)}{(4\alpha/q+1-2\alpha)(1+2\alpha)}>1,
\]
or equivalently
\[
F(\alpha):=(q-2)(2\alpha)^2-2(q\rho+1)(2\alpha)+2q\rho-q>0.
\]
Since $F(0)>0$ and $F(1/2)<0$, we obtain the condition
\[
0<\alpha<\alpha^*(q,\rho):=\frac{q\rho+1-\sqrt{(q(\rho-1)+1)^2+4q\rho}}{2(q-2)}<\frac12.
\]
In summary, under the parameter choice \eqref{eqn:thetachoice}, the  mean squared error bound \eqref{eqn:errorbound} decays at a nearly $O(n^{-1-2\alpha^*(q,\rho)})$ rate with an implied constant independent of $s$. In the limit $\rho q\to\infty$, this yields a dimension-independent error rate of nearly \(O(n^{-2})\).

% \He{If $q$ or $\rho$ can be arbitrarily large, we attain a dimension-independent mean squared error rate of nearly $O(n^{-2})$.}
\end{remark}

\section{Numerical experiments}\label{sec:numer}
To test our method, we consider standard Gaussian integrals, in which $\varphi(x)=\exp(-x^2/2)/\sqrt{2\pi}$, and test functions of the form
$$f(\bsx)=\prod_{j=1}^s \left( 1+j^{-2} g(x_j) \right) \text{ for } g(x_j)=\sqrt{1-2M}\exp(M x^2_j)-1.$$
We require $M<0.5$ so that $f\in L^1(\real^s,\varphi)$.
Because $\int_\real g(x)\varphi(x)\rd x=0$, a straightforward calculation using \eqref{eqn:anovadecomp2} shows $f_{v}=\prod_{j\in v} j^{-2}g(x_j)$. Because $ g \in W^{1,q}_{\mix}(\real,\varphi)$ for $q\in (1,1/(2M))$, $f\in W^{1,q}_{\mix}(\real^s,\varphi)$ and 
\begin{align*}
    \Vert f_v\Vert_{W^{1,q}_{\mix}(\real^s,\varphi)}=&\Big(\prod_{j\in v} j^{-2}\Big)\Big(\sum_{v'\subseteq v}  \Vert \partial^{v'} \prod_{j\in v}g(x_j)\Vert^q_{L^{q}(\real,\varphi)}\Big)^{1/q}\\
    =&\Big(\prod_{j\in v} j^{-2}\Big)\Vert g\Vert^{|v|}_{W^{1,q}_{\mix}(\real,\varphi)} .
\end{align*}
Since $f_\emptyset=1$, we have $\Vert f \Vert_q= \sup_{v\subseteq 1{:}s}\Vert f_v \Vert_{W^{1,q}_{\mix}(\real^s,\varphi)}\geq 1$, and 
\begin{equation*}
 \gamma_v=\Vert f\Vert^{-1}_{q} \Vert f_v \Vert_{W^{1,q,\infty}_{\mix}(\real^s,\varphi)}\leq \prod_{j\in v} \Gamma_j,\quad \text{where}\quad  \Gamma_j=j^{-2}\Vert g\Vert_{W^{1,q}_{\mix}(\real,\varphi)}.  
\end{equation*}
Moreover, $\Vert f \Vert_q$ is uniformly bounded over all $s\ge 1$.
Hence, Remark~\ref{rmk:tractability} yields a dimension-independent convergence rate for boundary-damping IS.

We first compare estimators \eqref{eqn:muhat} under different choices of $T_j$. We use the randomized lattice \(Q_{p,\bsz,\bsU}=\{\bsu_i\}_{i=0}^{p-1}\) from Theorem~\ref{thm:lattice} as the common quadrature rule, with smoothness parameter \(\alpha=1/2\) and weights \(\gamma_j=j^{-1/2}\). We draw \(p\) uniformly from \(\mathcal P_n\) given by \eqref{eqn:Pn}, where \(n=2^m\) with \(m=1,\dots,20\). For each \(n\), we generate \(80\) independent randomized lattices, and all estimators are evaluated on the same \(80\) lattices. The root mean squared error (RMSE) of each estimator \(\widehat{\mu}(f)\) is then estimated from these \(80\) runs.

Figure~\ref{fig:m00_comparison} shows the simulation results for $M=0$ and $s=5$ or $30$. In this case, $f(\bsx)=1$ and the usual inversion method with $T_j(u)=\Phi^{-1}(u)$ integrates $f$ exactly. We compare the performance of four options for $T_j$: 
\begin{itemize}
    \item Option 1: $T_j(u)$ in \eqref{eqn:Tj} with $\theta_j=0.1$.
    \item Option 2: $T_j(u)$ in \eqref{eqn:Tj} with $\theta_j=0.1/j^2$.
    \item Option 3: $T_j(u)=-\cot(\pi u)$.
    \item Option 4: $T_j(u)=a u-a(1-u)$.
\end{itemize}
Option 3 is the inverse CDF of a Cauchy distribution (also called the M{\"o}bius-transfor-mation in \cite{suzuki2025mobius}). Option 4 is the truncation method with  $a=\sqrt{2\log n}$ suggested by \cite[Theorem 1b]{nuyens2023scaled} (we also tested $a=2\sqrt{\log n}$ suggested by \cite{dick2018optimal} and observed larger RMSEs). When $s=5$, all methods except Option 4 achieve a nearly $O(n^{-0.8})$ convergence rate, with Option 2 performing slightly better than Option 1 and Option 3. Option 4 seems already suffering from the dimensionality. When $s=30$, all methods are suffering from the high dimensionality, with Option 2 still maintaining a convergence rate close to $O(n^{-0.65})$. 

\begin{figure}[h!] % The 'h!' option tells LaTeX to "place the figure HERE if possible"
    \centering % Centers the entire figure content on the page
    
    % First subfigure (M=0.0)
    \begin{subfigure}[b]{0.48\textwidth} % 'b' aligns the subfigures by their bottom edge
        \centering
        \includegraphics[width=\textwidth]{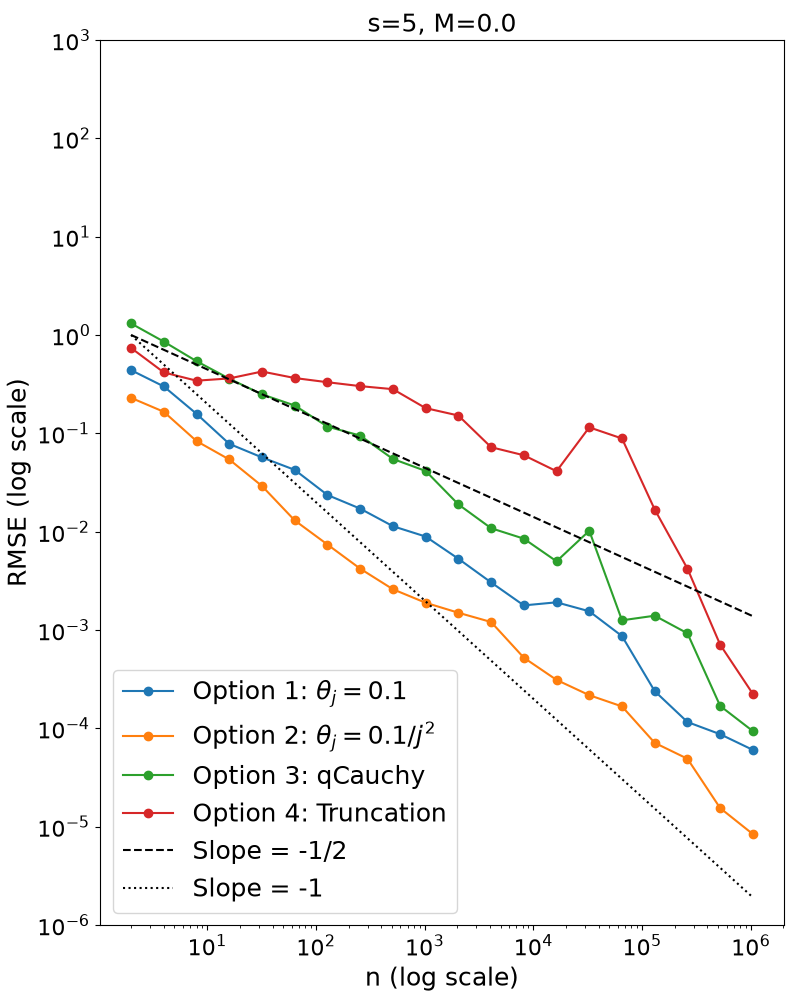}
        \caption{$s=5, M=0$}
        \label{fig:s5_m00} % A label for referencing this subfigure
    \end{subfigure}
    \hfill % Adds a horizontal flexible space between the two subfigures
    % Second subfigure (M=0.3)
    \begin{subfigure}[b]{0.48\textwidth} % Set width to < 0.5\textwidth to fit on one line
        \centering
        \includegraphics[width=\textwidth]{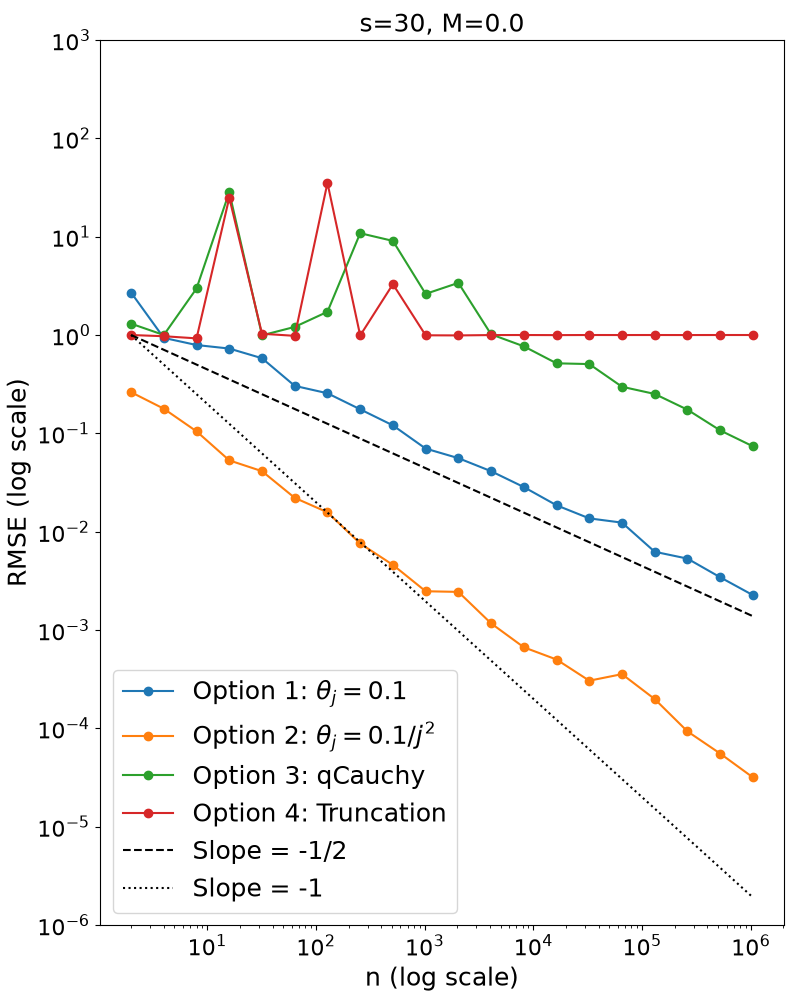}
        \caption{$s=30, M=0$}
        \label{fig:s30_m00}
    \end{subfigure}
    
    \caption{Comparison of RMSEs for $M=0$ with $s=5$ or $30$. The first four legend labels correspond to Options 1-4, and the two reference lines are  proportional to $n^{-1/2}$ and  $n^{-1}$, respectively. }
    \label{fig:m00_comparison} % A label for referencing the entire figure
\end{figure}

Figure~\ref{fig:m03_comparison} shows the simulation results for $M=0.3$ and $s=5$ or $30$. In this case, $f\in W^{1,q}_{\mix}(\real^s,\varphi)$ for $q\in (1,5/3)$, and the plain Monte Carlo method for estimating $\mu(f)$ has an infinite variance. We again compare the performance of Options 1-4. Since $f(\bsx)$ grows like $\exp(0.3\Vert \bsx\Vert^2_2)$ as  $\Vert \bsx\Vert_2\to\infty$, Theorem 1 of \cite{nuyens2023scaled} suggests setting $a=\sqrt{5\log n}$ in Option 4. In addition to the previous four options, we also compare the usual inversion method (without importance sampling):
\begin{itemize}
    \item Option 5: $T_j(u)=\Phi^{-1}(u).$
\end{itemize}
By Remark~\ref{rem:RQMCrate}, the asymptotic RMSE convergence rate of Option 5 is close to $O(n^{-0.4})$, consistent with the simulation results. Options 1–4 perform comparably to the case \(M=0\), indicating that they can handle severe boundary growth. When $s=30$, Option 2 still achieves a convergence rate close to the theoretically optimal rate $O(n^{-1/2-\alpha^*(q,\rho)})$, where $\alpha^*(q,\rho)$ is given by \eqref{eq:alpha bound} and has the limiting value $7/26$ as $q\to 5/3$.

\begin{figure}[h!] % The 'h!' option tells LaTeX to "place the figure HERE if possible"
    \centering % Centers the entire figure content on the page
    
    % First subfigure (M=0.0)
    \begin{subfigure}[b]{0.48\textwidth} % 'b' aligns the subfigures by their bottom edge
        \centering
        \includegraphics[width=\textwidth]{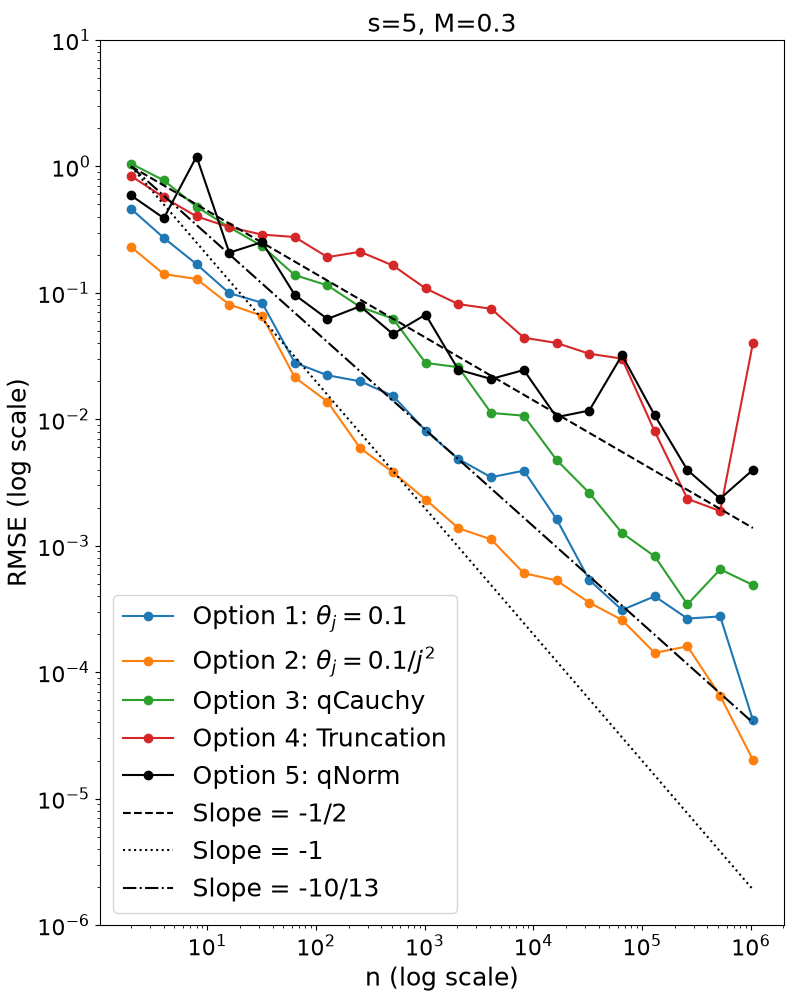}
        \caption{$s=5$, $M=0.3$}
        \label{fig:s5_m03} % A label for referencing this subfigure
    \end{subfigure}
    \hfill % Adds a horizontal flexible space between the two subfigures
    % Second subfigure (M=0.3)
    \begin{subfigure}[b]{0.48\textwidth} % Set width to < 0.5\textwidth to fit on one line
        \centering
        \includegraphics[width=\textwidth]{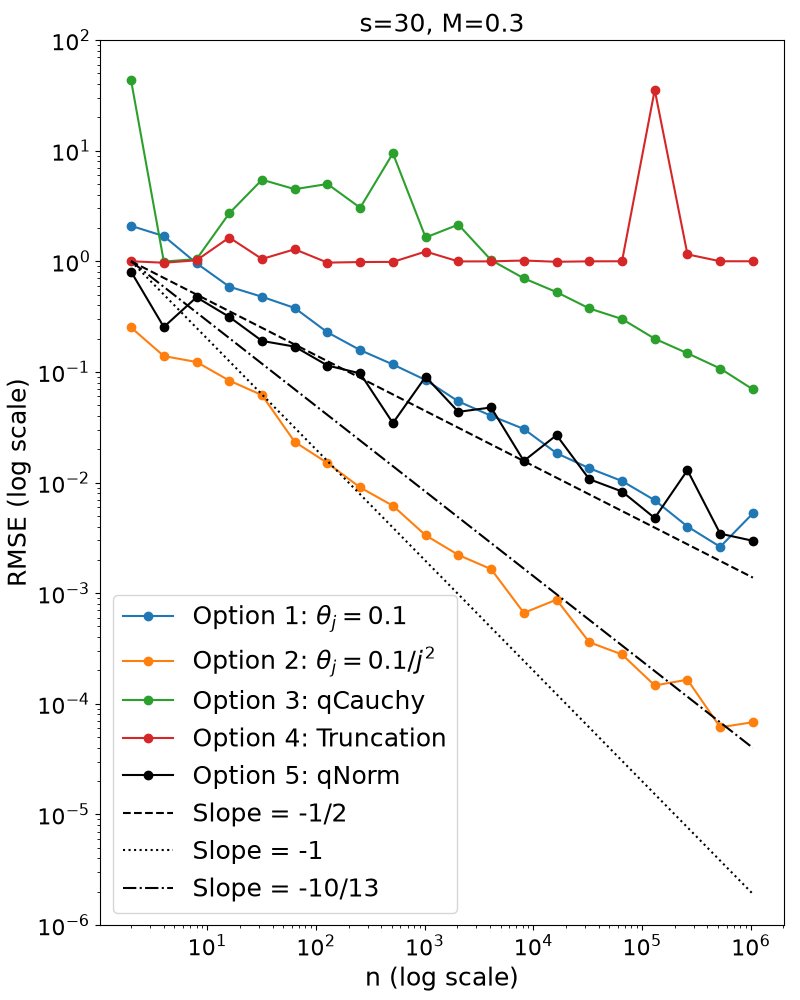}
        \caption{$s=30$, $M=0.3$}
        \label{fig:s30_m03}
    \end{subfigure}
    
    \caption{Comparison of RMSEs for $M=0.3$ with $s=5$ or $30$. The first five legend labels correspond to Options 1-5, and the three reference lines are  proportional to $n^{-1/2}$, $n^{-1}$, and $n^{-10/13}$, respectively.  }
    \label{fig:m03_comparison} % A label for referencing the entire figure
\end{figure}

Our final experiment studies how the choice of $\theta_j$ affects the performance of boundary-damping IS. We set $M=0.25$ so that $f\in W^{1,q}_{\mix}(\real^s,\varphi)$ for $q\in (1,2)$. Since $\rho=2$, Remark~\ref{rmk:tractability} suggests choosing $\theta_j=\theta_0 j^{-4}$  with $\theta_0\in (0,1/2]$. Moreover, $\alpha^*(q,\rho)$ in \eqref{eq:alpha bound} can be arbitrarily close to $0.3$, so the resulting dimension-independent RMSE rate is nearly \(O(n^{-0.8})\). For convenience, we take the limiting values \(\alpha=0.3\) and \(\varepsilon=0\) in the numerical implementation of Corollary~\ref{thm:qmcvar}, and use randomized lattices \(Q_{p,\bsz,\bsU}\) with  \(\gamma_j=j^{-0.3}\).

We compare the following three options:
\begin{itemize}
    \item Option 6: $T_j(u)$ in \eqref{eqn:Tj} with $\theta_j=0.1/j^2$.
    \item Option 7: $T_j(u)$ in \eqref{eqn:Tj} with $\theta_j=0.1/j^4$.
    \item Option 8: $T_j(u)$ in \eqref{eqn:Tj} with $\theta_j=0.1/j^6$.
\end{itemize}
We also use the usual inversion method Option 5 as a baseline. The results for $s=128$ are shown in Figure~\ref{fig:m025}. We see Options 6-8 significantly outperform the baseline, indicating our boundary-damping IS successfully accelerates the convergence in high-dimensional settings. We also observe Option 7 performs the best among Options 6-8, confirming our prediction. 

\begin{figure}[h!] % The 'h!' option tells LaTeX to "place the figure HERE if possible"
    \centering % Centers the entire figure content on the page
    \includegraphics[width=0.48\textwidth]{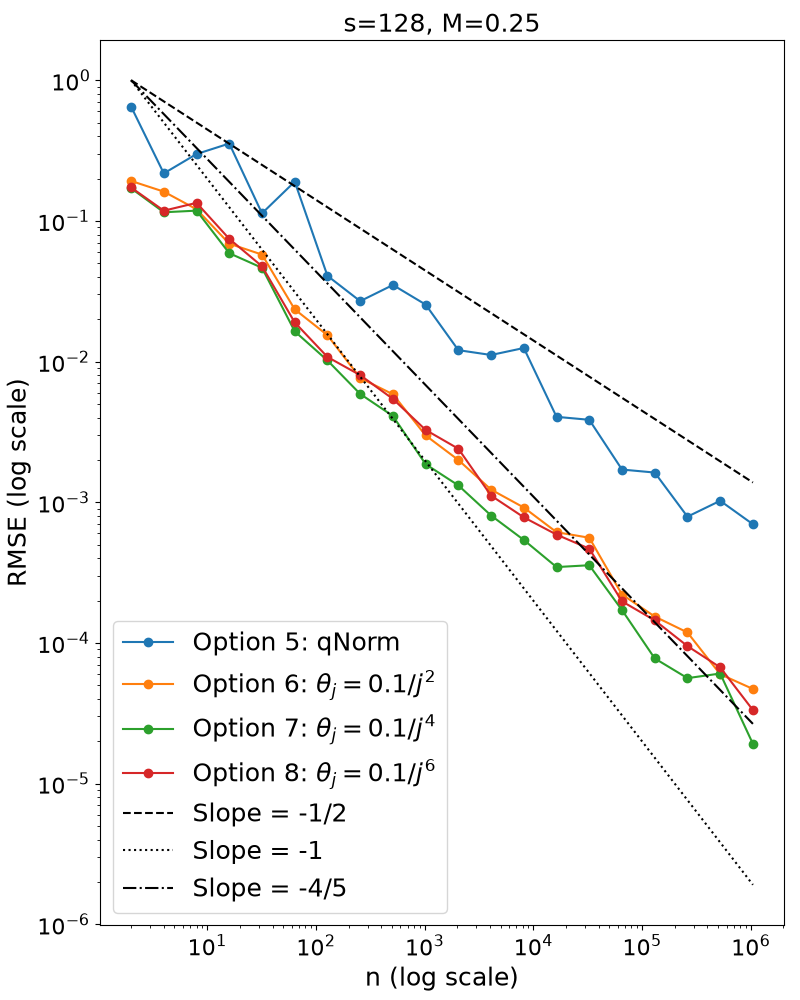}
    \caption{RMSEs for $M=0.25$ with $s=128$. The first four legend labels correspond to Options 5-8, and the three reference lines are  proportional to $n^{-1/2}$, $n^{-1}$ and $n^{-0.8}$, respectively.}
    \label{fig:m025} % A label for referencing the entire figure
\end{figure}

\section{Concluding remarks}\label{sec:conclusion}

In this paper, we have proposed a new class of IS methods suitable for QMC integration of functions with severe boundary growth. Both our theoretical bounds and simulation results demonstrate a significant improvement in convergence rates compared with existing approaches.

A practical difficulty in implementing boundary-damping IS is the choice of the parameter \(\theta_j\) for each dimension \(j\in 1{:}s\). The optimal choice \eqref{eqn:thetachoice} relies on prior knowledge of the dimension structure and the boundary growth of the integrand. When no such knowledge is available, we suggest choosing \(\theta_j\) to be summable so that \(\prod_{j=1}^\infty (1-\theta_j)^{-1}<\infty\). Lemma~\ref{lem:fwLq} then ensures that the IS-transformed integrand \(f^w\) has \(L^q\) norm bounded by \(\|f\|_{L^q(\real^s,\varphi)}\) up to a dimension-independent constant. For \(q=2\), this means that \(f^w\) is amenable both to Monte Carlo and to QMC rules that attain the Monte Carlo rate for \(L^2\) functions \cite{owen1998scrambling,suzuki2025coarse,ye2026}. Figure~\ref{fig:m025} further suggests that the performance of boundary-damping IS is robust to mild deviations from the optimal choice.

A further practical difficulty is the construction of the lattice rule from Theorem~\ref{thm:lattice}. Choosing the smoothness parameter \(\alpha\) and the weight parameters \(\bsgamma\) optimally also requires information about the dimension structure and the boundary growth. A possible solution is to use median lattice rules, which do not take \(\alpha\) and \(\bsgamma\) as inputs \cite{goda:lecu:2022,goda2024universal,goda2024simpleuniversal}. A full investigation is beyond the scope of this paper and is left for future work.

\begin{appendices}

\section{Proof of Lemma~\ref{lem:Sobolevimbedding}}\label{app:lem:Sobolevimbedding}

  Let $I(y, x, t)$ be the signed indicator function defined as $1$ if $y \leq t \leq x$, $-1$ if $x \leq t < y$, and $0$ otherwise. For almost every $\bsx,\bsy \in \itr^s$, we can iteratively apply the one-dimensional fundamental theorem of calculus along each coordinate and obtain the following multivariate representation for $f\in W^{1,1}_{\mix}(\itr^s)$:
    \begin{equation}\label{eq:multivariate_FTC}
        f(\bsx) = \sum_{v \subseteq 1{:}s} \int_{\itr^{|v|}} \partial^v f(\bst_v, \bsy_{v^c}) \prod_{j \in v} I(y_j, x_j, t_j) \rd \bst_v,
    \end{equation}
    where $(\bst_v, \bsy_{v^c})$ denotes the vector $\bsu \in \itr^s$ such that $u_j = t_j$ for $j \in v$ and $u_j = y_j$ for $j \notin v$, and the derivatives $\partial^v f$ are understood in the weak sense.

    Integrating both sides of \eqref{eq:multivariate_FTC} with respect to $\bsy$ over the unit cube, and noting that the left side is independent of $\bsy$, we obtain
    \begin{equation*}
        f(\bsx) = \sum_{v \subseteq 1{:}s} \int_{\itr^s}  \int_{\itr^{|v|}} \partial^v f(\bst_v, \bsy_{v^c}) \prod_{j \in v} I(y_j, x_j, t_j) \rd \bst_v  \rd \bsy.
    \end{equation*}
    By Fubini's theorem, we can change the order of integration. Let $\bsy = (\bsy_v, \bsy_{v^c})$ and $\rd \bsy = \rd \bsy_v \rd \bsy_{v^c}$. We change the variables of integration from $(\bst_v, \bsy)$ to $\bsu = (\bst_v, \bsy_{v^c}) \in \itr^s$ and $\bsy_v \in \itr^{|v|}$:
    \begin{equation*}
        f(\bsx) = \sum_{v \subseteq 1{:}s} \int_{\itr^s} \partial^v f(\bsu) \left( \prod_{j \in v} \int_0^1 I(y_j, x_j, u_j) \rd y_j \right) \rd \bsu.
    \end{equation*}
    Taking the absolute value on both sides, and observing that $\left| \int_0^1 I(y_j, x_j, u_j) \rd y_j \right| \leq 1$, we can bound
    \begin{equation*}
        |f(\bsx)| \leq \sum_{v \subseteq 1{:}s} \int_{\itr^s} |\partial^v f(\bsu)| \rd \bsu = \Vert f \Vert_{W^{1,1}_{\mix}(\itr^s)}.
    \end{equation*}
    Taking the essential supremum over $\bsx \in \itr^s$ gives $\Vert f \Vert_{L^\infty(\itr^s)} \leq \Vert f \Vert_{W^{1,1}_{\mix}(\itr^s)}$. 

\section{Proof of Lemma \ref{lem:anova}}\label{app:lem:anova}
We first prove that \(\partial^{v}P_{w}(f)=P_{w}(\partial^{v}f)\) in the weak sense. Let $\partial_i f$ denote $\partial f/\partial x_i$. It suffices to show that \(\partial_i P_j(f)=P_j(\partial_i f)\) if $i\neq j$. The general case follows by induction.

Let $\psi:\real^s\to \real$ be a compactly supported smooth function. We need to show that
\begin{equation}\label{eqn:interchange}
    \begin{aligned}
&\int_{\real^s}\psi(\bsx)\left(\int_{\real}\left(\partial_i f\right)(y_j,\bsx_{-j})\varphi(y_j)\rd y_j\right)\rd \bsx\\
    =&-\int_{\real^{s}}(\partial_i\psi)(\bsx)\left(\int_{\real}f(y_j,\bsx_{-j})\varphi(y_j)\rd y_j\right)\rd \bsx.  
    \end{aligned}
\end{equation}

Since $\psi$ is compactly supported, we can find a compact set $K\subseteq \real$ such that $\psi(\bsx)=0$ outside $K^s$.
Because $\partial_i f\in L^{1}(\real^s,\varphi)$,
\begin{equation*}
    \begin{aligned}
     &  \int_{\real^s}\int_{\real}|\psi(\bsx)|\left(\partial_i f\right)(y_j,\bsx_{-j})\varphi(y_j)\rd y_j\rd \bsx\\
\leq & \frac{\max_{\bsx\in K^s  }|\psi(\bsx)|}{(\inf_{x\in K  }\varphi(x))^s }\int_{\real^{s}}\int_{\real}\left|\left(\partial_i f\right)(y_j,\bsx_{-j})\right| \varphi(y_j)\prod_{r=1}^s\varphi(x_r)\rd y_j\rd \bsx<\infty. 
    \end{aligned}
\end{equation*}
Hence, we can apply Fubini's theorem to the left hand side of \eqref{eqn:interchange} and obtain
\begin{equation*}
    \begin{aligned}
&\int_{\real^s}\psi(\bsx)\left(\int_{\real}\left(\partial_i f\right)(y_j,\bsx_{-j})\varphi(y_j)\rd y_j\right)\rd \bsx\\
        =&\int_{\real^{s-1}}\int_{\real}g(y_j,\bsx_{-i})\varphi(y_j)\rd y_j\rd \bsx_{-i},
    \end{aligned}
\end{equation*}
where 
$$g(y_j,\bsx_{-i})=\int_{\real}(\partial_i f)(y_j,x_i,\bsx_{\{i,j\}^c})\psi(x_i,\bsx_{-i})\rd x_i.$$

By the ACL-characterization of Sobolev functions \cite[Chapter 1.1.3]{mazya2011sobolev}, after modifying $f$ on a null set if necessary, the slice \(x_i\longmapsto f(y_j,x_i,\bsx_{\{i,j\}^c})\)
is absolutely continuous for almost every
$(y_j,\bsx_{\{i,j\}^c})\in \real^{s-1}$, and its classical derivative is almost everywhere equal to \((\partial_i f)(y_j,x_i,\bsx_{\{i,j\}^c}).\)
Hence integration by parts in $x_i$ gives
\[
g(y_j,\bsx_{-i})=
-\int_{\real}
f(y_j,x_i,\bsx_{\{i,j\}^c})(\partial_i\psi)(x_i,\bsx_{-i})
\,dx_i,
\]
for almost every $(y_j,\bsx_{-i})\in \real\times \real^{s-1}$. Since $f\in L^{1}(\real^s,\varphi)$, another application of Fubini's theorem yields \eqref{eqn:interchange}.

Next we show that $P_{w}(\partial^{v}f)\in L^q(\real^s,\varphi)$. The case $q=\infty$ is immediate.
For $q<\infty$, Jensen's inequality gives, 
\[
\begin{aligned}
|P_{w}(\partial^{v}f)(\bsx)|^q
&=
\left|
\int_{\real^{|w|}}
(\partial^{v}f)(\bsy_w,\bsx_{w^c})
\prod_{j\in w}\varphi(y_j)\,d\bsy_w
\right|^q\\
&\le
\int_{\real^{|w|}}
|(\partial^{v}f)(\bsy_w,\bsx_{w^c})|^q
\prod_{j\in w}\varphi(y_j)\,d\bsy_w.
\end{aligned}
\]
Multiplying by $\prod_{j\notin w}\varphi(x_j)$ and integrating with
respect to $\bsx_{w^c}$ yields
\[
\|P_{w}(\partial^{v}f)\|^q_{L^q(\real^s,\varphi)}
\le
\|\partial^{v}f\|^q_{L^q(\real^s,\varphi)}
<\infty.
\]

Finally, we notice that $\partial^vP_w(f)=0$ if $v\cap w\neq\emptyset$ since $P_w(f)$ is constant in $\bsx_w$. Together with $\partial^{v}P_{w}(f)=P_{w}(\partial^{v}f)\in L^q(\real^s,\varphi)$ if $v\cap w=\emptyset$, this shows that $P_w(f)\in W^{1,q}_{\mix}(\real^s,\varphi)$ for every $w\in 1{:}s$. 
By \eqref{eqn:anovadecomp2}, $f_v$ is a linear combination of $P_{ (v')^c}(f)$ for $v'\subseteq v$, so $f_v\in W^{1,q}_{\mix}(\real^s,\varphi)$.

\section{Proof of Lemmas~\ref{lem:rhoTbound} and \ref{lem:dwjbound}}\label{app:lem:rhoTbound}

Our proof requires the following lemma.

\begin{lemma}\label{eqn:etabound}
Over $u\in (0,1/2]$,
$$|\eta'(u)|\leq u^{-2} \eta(u)  \quad\text{and} \quad \int_{0}^u\eta(t)\rd t\geq u^2 \eta(u). $$
\end{lemma}
\begin{proof}
    The two inequalities follow directly from
    $$ |\eta'(u)| = |u^{-2} - 2u^{-1}| 2^{-3} u^{-2}\exp(2 - u^{-1}) \le u^{-2} \eta(u)$$
    and
    \begin{align*}
              \int_0^u  \eta(t)\rd t = 2^{-3}\exp(2 - u^{-1}) = u^{2} \eta(u).  
    \end{align*}
\end{proof}

\begin{proof}[Proof of Lemma~\ref{lem:rhoTbound}]
    By Assumption~\ref{assump:rho} and \eqref{eqn:Tj}, for every \(\varepsilon>0\),
    \begin{equation}\label{eqn:phiTjbound}
        (\varphi\circ T_{\theta})(u)
        \geq c_\varepsilon\Big(\Phi\circ\Phi^{-1}\Big(\int_0^u w_{\theta}(t)\rd t\Big)\Big)^{1+\varepsilon}
        = c_\varepsilon\Big(\int_0^u w_{\theta}(t)\rd t\Big)^{1+\varepsilon}.
    \end{equation}
    We first estimate the integral of \(w_\theta\). When \(u\in(0,\theta/2]\),
    \[
    \int_0^u w_{\theta}(t)\rd t
    =(1-\theta)^{-1}\int_0^{u}\eta(t/\theta)\rd t
    =(1-\theta)^{-1}\theta \int_0^{u/\theta}\eta(t)\rd t.
    \]
    Lemma~\ref{eqn:etabound} implies that if \(u\in(0,\theta/2]\),
    \[
    \int_0^u w_{\theta}(t)\rd t
    \geq \theta^{-1} u^{2}w_{\theta}(u).
    \]
    When \(u\in(\theta/2,\theta)\), since \(w_{\theta}(\theta/2)=(1-\theta)^{-1}/2\geq w_{\theta}(u)/2\),
    \[
    \int_0^u w_{\theta}(t)\rd t
    \geq \int_0^{\theta/2} w_{\theta}(t)\rd t
    \geq \theta^{-1}(\theta/2)^2w_{\theta}(\theta/2)
    \geq (\theta/8) w_{\theta}(u).
    \]
    Finally when \(u\in[\theta,1/2]\),
    \[
    \int_0^u w_{\theta}(t)\rd t
    = w_{\theta}(u)(u-\theta/2).
    \]
    Combining these estimates with \eqref{eqn:phiTjbound} yields, for some constant \(C_\varepsilon>0\),
    \begin{equation}\label{eqn:phiTtheta-new}
        (\varphi\circ T_{\theta})(u)\geq
        \begin{cases}
            C_\varepsilon \big(\theta^{-1}u^{2}w_{\theta}(u)\big)^{1+\varepsilon}, &\text{ if } u\in(0,\theta), \\[2mm]
            C_\varepsilon \big((u-\theta/2)w_{\theta}(u)\big)^{1+\varepsilon}, &\text{ if }  u\in[\theta,1/2].
        \end{cases}
    \end{equation}

    We now prove the \(L^q\) bound by establishing the \(L^1\) and \(L^\infty\) cases separately and then interpolating. Set
    \[
    g(u):=\frac{w_\theta(u)^\lambda}{\varphi\circ T_\theta(u)}.
    \]

For \(u\in(0,\theta)\), using the firsst case of \eqref{eqn:phiTtheta-new},
\[
g(u)\le C_\varepsilon^{-1}\theta^{1+\varepsilon}u^{-2(1+\varepsilon)}w_\theta(u)^{\lambda-1-\varepsilon}.
\]
Substituting \(u=\theta t\), \(t\in(0,1)\), gives
\[
g(\theta t)\le C_\varepsilon^{-1}\theta^{-1-\varepsilon}t^{-2(1+\varepsilon)}\eta(t)^{\lambda-1-\varepsilon}
=:C_\varepsilon^{-1}\theta^{-1-\varepsilon}h(t),
\]
where
\[
h(t)=t^{-2(1+\varepsilon)}\eta(t)^{\lambda-1-\varepsilon}.
\]
Since \(\eta(t)\sim t^{-2}\exp(-t^{-1})\) as \(t\to0^+\), the function \(h\) is bounded and integrable on \((0,1)\), with bounds depending only on \(\lambda\) and \(\varepsilon\).

For \(u\in[\theta,1/2]\), using the second case of \eqref{eqn:phiTtheta-new} and \(w_\theta(u)\le 2\),
\[
g(u)\le C_\varepsilon^{-1}2^{\lambda-1-\varepsilon}(u-\theta/2)^{-(1+\varepsilon)}.
\]
By symmetry, it suffices to bound \(g\) on \((0,1/2]\). We obtain
\[
\|g\|_{L^\infty(\itr)}
\le C_\varepsilon^{-1} \max\left(\theta^{-1-\varepsilon}\sup_{t\in(0,1)}h(t),\;2^{\lambda-1-\varepsilon}(\theta/2)^{-1-\varepsilon}\right)
\le C_{\lambda,\varepsilon}\theta^{-1-\varepsilon},
\]
and
\[
\|g\|_{L^1(\itr)}
\le 2C_\varepsilon^{-1}\left(\theta^{-\varepsilon}\int_0^1 h(t)\rd t
+2^{\lambda-1-\varepsilon}\int_\theta^{1/2}(u-\theta/2)^{-(1+\varepsilon)}\rd u\right)\leq C'_{\lambda,\varepsilon}\theta^{-\varepsilon}.
\]

  Interpolating the two bounds completes the proof.
\end{proof}

\begin{proof}[Proof of Lemma~\ref{lem:dwjbound}]
    Differentiating \eqref{eqn:wtheta} gives
    \[
    w'_{\theta}(u)=\begin{cases}
    \theta^{-1}(1-\theta)^{-1}\eta'(u/\theta), &\text{ if } u\in (0,\theta/2], \\
    \theta^{-1}(1-\theta)^{-1}\eta'(1-u/\theta), &\text{ if } u\in (\theta/2,\theta), \\
    0, &\text{ if } u\in [\theta,1/2].
\end{cases}
    \]
    When \(u\in (0,\theta/2]\), Lemma~\ref{eqn:etabound} implies
    \[
    |w'_\theta(u)|\leq 2\theta u^{-2} w_\theta(u).
    \]
    When \(u\in (\theta/2,\theta)\), using \(w_\theta(u)\ge w_\theta(\theta/2)=(1-\theta)^{-1}/2\),
    \[
    |w'_\theta(u)|\leq \theta^{-1}(1-\theta)^{-1}\sup_{t\in [0,1/2]}\eta'(t)
    \leq C \theta^{-1} w_\theta(u)
    \leq C \theta u^{-2} w_\theta(u).
    \]
    Therefore, with a possibly larger absolute constant \(C\),
    \begin{equation}\label{eqn:wprime-bound}
        |w'_\theta(u)|\leq C \theta u^{-2} w_\theta(u),
        \qquad u\in(0,\theta).
    \end{equation}

  We now prove the \(L^q\) bound. Set
    \[
    g(u):=\frac{|w'_\theta(u)|}{w_\theta(u)^\lambda}.
    \]

   It suffices to consider \(u\in(0,\theta)\). Using \eqref{eqn:wprime-bound} and letting $t=u/\theta\in (0,1)$,
    \[
    g(u)\le  C \theta u^{-2} w_\theta(u)^{1-\lambda} 
    =:C \theta^{-1} h(t),
    \]
    where
    \[
    h(t)=t^{-2}\eta(t)^{1-\lambda}.
    \]
    Since \(\eta(t)\sim t^{-2}\exp(-t^{-1})\) as \(t\to0^+\), the function \(h\) is bounded and integrable on \((0,1)\), with bounds depending only on \(\lambda\). Hence
    \[
    \|g\|_{L^\infty(\itr)}
    \le C \theta^{-1}\sup_{t\in(0,1)}h(t)
    \le C_\lambda \theta^{-1},
    \]
    and
    \[
    \|g\|_{L^1(\itr)}
    = 2\theta \int_0^1 g(\theta t)\rd t
    \le C \int_0^1 h(t)\rd t
    \le C_{\lambda,q} \theta^{-1+1/q}
    \]

     Interpolating the two bounds completes the proof.
\end{proof}
\end{appendices}
  
\section*{Acknowledgments}
We thank Yuya Suzuki  for helpful comments. 

%The first author was supported by the Austrian Science Fund (FWF) Project DOI 10.55776/P34808, and the National Natural Science Foundation of China (No. 12601830). The second author was supported by the Soci\'et\'e G\'en\'erale Chair ``Risques Financiers''.  The third author was supported by the National Natural Science Foundation of China (No. 12671500), and the Guangdong Basic and Applied Basic Research Foundation (No. 2025A1515011888).

\bibliography{qmc}% common bib file

\end{document}